\documentclass[11pt,a4wide]{article}

\usepackage{multirow}
\usepackage{algorithmic}
\usepackage{amsmath}
\usepackage{amsfonts}
\usepackage{amsthm}
\usepackage[english]{babel} 
\usepackage{dsfont}
\usepackage{amssymb}
\usepackage{graphicx}        
\usepackage{cite}                            
\usepackage{subcaption}
\usepackage{epstopdf}
\usepackage{epsfig}
\usepackage{subfloat}
\usepackage{float}
\usepackage[thinlines]{easytable}
\usepackage{array}
\usepackage{mathtools}
\usepackage{stmaryrd}
\usepackage{bm}
\usepackage{hyperref}
\hypersetup{
    colorlinks=true,
    linkcolor=blue,
    filecolor=magenta,      
    urlcolor=cyan,
    pdftitle={Overleaf Example},
    pdfpagemode=FullScreen,
    }
\def\dist{\textup{dist}}

\usepackage{xcolor,colortbl}
\definecolor{g}{rgb}{0.68,1,0.18}

\usepackage{booktabs}
\usepackage{url}
\usepackage[font=small,labelfont=bf,textfont=it,indention=.1cm,width=1.1\textwidth]{caption}

\usepackage[capitalise]{cleveref}
\usepackage{comment}

\allowdisplaybreaks
\numberwithin{equation}{section}

\newtheorem{remark}{Remark}[section]
\newtheorem{definition}{Definition}[section]
\newtheorem{lemma}{Lemma}[section]
\newtheorem{theorem}{Theorem}[section]
\newtheorem{corollary}{Corollary}[section]

\usepackage[linesnumbered,algoruled,boxed,lined]{algorithm2e}

\def\RR{\mathbb R}

\def\ve{\varepsilon}

\def\be{\begin{equation}}
\def\ee{\end{equation}}
\def\bea{\begin{eqnarray}}
\def\eea{\end{eqnarray}}

\newtheorem{assumption}{Assumption}[section]

\title{An adaptive consensus based method for multi-objective optimization with uniform Pareto front approximation}

\author{Giacomo Borghi\footnote{RWTH Aachen University, Institute for Geometry and Applied Mathematics, Aachen, Germany (borghi@eddy.rwth-aachen.de, herty@igpm.rwth-aachen.de)}
\and Michael Herty \footnotemark[1] 
\and Lorenzo Pareschi\footnote{University of Ferrara, Department of Mathematics and Computer Science \& Center for Modelling Computing and Statistics, Ferrara, Italy (lorenzo.pareschi@unife.it)}}

\begin{document}
\maketitle

\begin{abstract} 
In this work we are interested in stochastic particle methods for  multi-objective optimization. The problem is formulated using parametrized, single-objective sub-problems which are solved simultaneously. To this end  a consensus based multi-objective optimization method on the search space combined with an additional heuristic strategy to adapt parameters during the computations is proposed. The adaptive strategy aims to distribute  the particles uniformly over the image space by using energy-based measures to quantify the diversity of the system. The resulting metaheuristic algorithm is mathematically analyzed using a mean-field approximation and convergence guarantees towards optimal points is rigorously proven. In addition, a gradient flow structure in the parameter space for the adaptive method is revealed and analyzed.  Several numerical experiments shows the validity of the proposed stochastic particle dynamics and illustrate the theoretical findings.
\end{abstract}

\smallskip


\tableofcontents

\section{Introduction}
Motivated by the problem in which two or more objectives must be considered at the same time, even though they may conflict with each other, in this work we are interested the design of stochastic algorithms for multi-objective optimization. This type of problem is commonly found in everyday life, for example, in physics, engineering, social sciences, economy, biology, and many others^^>\cite{deb2001multi,Hwang79,Pardalos2018,jahn2004vector,eichfelder2021twenty}. Investing in the financial market while maximizing profit and minimizing risk or building a vehicle while maximizing performance and minimizing fuel consumption and pollutant emissions are examples of multi-objective optimization problems.

From a mathematical viewpoint the problem can be formulated through a variable $x\in \RR^d$ describing a possible decision and assuming that $g_i(x)$ is the $i$-th objective for $i=1, \dots, m$, with $m \in \mathbb{N}$ being the total number of objectives. A multi--objective problem then requires to solve for a decision $x$ 
\begin{equation}
\min_{x \in \RR^d} g(x)
\label{eq:mop}
\end{equation}
where $g(x) = (g_1(x) , \dots, g_m(x))^\top$. A solution to \eqref{eq:mop} corresponds to several optimal decisions. Here, we consider optimality in the sense of Pareto \cite{jahn2004vector}, i.e.,  no objective  can  be improved  without necessarily degrade another objective. Without additional information about subjective preferences, there may be a (possibly infinite) number of Pareto optimal solutions, all of which are considered equally good. 
Therefore, the optimization tasks consist of providing  a set of optimal decisions. To this end, it is also desirable to have a \textit{diverse} set, that is, addressing the problem not only by optimizing fitness, but also by aiming to cover a variety of user-defined features of interest, in order to  best describe the (possibly) broad set  of optimal decisions.

Several methods have been proposed to numerically solve \eqref{eq:mop} and, as for single-objective optimization, they typically belong to either the class of metaheuristic algorithms or mathematical programming methods \cite{Sergeyev2018}.
Among metaheuristics \cite{talbi2009meta}, multi-objective evolutionary algorithms \cite{deb2001multi}, such as NSGA-II \cite{deb2002nsga2} and MOEA/D \cite{zhang2008moead}, have gained popularity among practitioners due to their flexibility and ease of use. At the same time, they usually lack of convergence analysis compared to mathematical programming methods. For more details on mathematical programming methods and evolutionary algorithms in multi-objective optimization   we refer to the recent surveys \cite{eichfelder2021twenty, coello2020survey}. 

We are interested in a particular class of stochastic particle optimization methods, called consensus-based optimization (CBO), which has recently gained popularity due to the use of mean-field techniques that can provide them with a rigorous mathematical foundation. Such methods consider interacting particle systems described by stochastic differential equations (SDEs) that combine a drift towards the estimated minimum and random exploration of the search space^^>\cite{pinnau2017consensus,carrillo2018analytical,fornasier2021consensusbased,carrillo2019consensus,fornasier2022aniso,benfenati2021binary,totzeck2020}. These approaches have been extended also to optimization problems over hypersurfaces \cite{fhps20-2,fhps20-2b,fornasier2020hypersurfaces}, constrained optimization  \cite{borghi2021constrained,carrillo2021constrained} and multi-objective optimization \cite{borghi2022multi}. From a mathematical viewpoint, this class of metaheuristic methods is inspired by the corresponding mean-field dynamics based on particle swarming and multi-agent social interactions, which have been widely used to study complex systems in life sciences, social sciences and economics \cite{pareschi13,MR3274797,defrli13,Prigogine1977self,Vicseck}. These techniques have proven fruitful to demonstrate convergence towards a global minimum for single-objective problems, not only in the case of CBO methods, but also for the popular Particle Swarm Optimization (PSO) algorithm \cite{grassi2021from,huang2022PSO}, thus paving the way to provide a mathematical foundation for other metaheuristics.

In the same spirit, the authors proposed in \cite{borghi2022multi} a multi-objective optimization algorithm (M-CBO) by prescribing a CBO-type dynamics among several particles making use of a scalarization strategy. Scalarization strategies are a common tool in multi-objective optimization \cite{jahn2004vector} as they allow to translate problem \eqref{eq:mop} into a set of parametrized single-objective problems, which can be solved simultaneously in the case of particle-based optimization methods.  
In this paper, we provide a convergence analysis for the method based on the mean-field description of the M-CBO dynamics. Furthermore, we improve the method in order to capture with uniform accuracy the shape of the Pareto front. This is done by  iteratively updating parameters of the method to minimize  specific diversity measures. Mathematically, this last feature is achieved by enlarging the phase space of the particles. A detailed analysis of the extended model is also presented  by studying a mean-field approximation of the particle dynamics which allows to recover convergence guarantees towards optimal points and also underline a gradient-flow structure in the space of the parameters. 

Recently, energy-based diversity measures have gained popularity in the multi-objective evolutionary optimization community due to their flexibility, scalability and theoretical properties \cite{coello2020survey}. In this formulation, a set of decision is diverse if it corresponds to a minimal configuration of a suitable two-body energy potential, breaking the problem down into finding such configurations. In the  proposed algorithm, we obtain this by  inserting a Vlasov-type dynamics in the space of the parameters.  We prove that the particle dynamics can be written in the more general framework of non-local interaction equations over bounded domains. The later topic has been recently investigated e.g. in  \cite{carrillo2016nonlocal,fatecau2017swarm,fatecau2019diffusion,patacchini2022nonlocal}.

The rest of the paper is organized as follows. In \cref{sec:2} we formally introduce the concept of optimality for \eqref{eq:mop} and present the scalarization strategy. Next,
in \cref{sec:3} we illustrate the particle dynamics both in the search space and in the space of parameters. \cref{sec:4} is devoted to the mathematical analysis of the system evolution using a mean-field description. Finally, in \cref{sec:5}
numerical examples on convex and non-convex as well as disjoint Pareto fronts are presented which confirm the theoretical results as well as the performance of the new method. Some concluding remarks are discussed in the last section.

\section{Problem definition and scalarization}
\label{sec:2}

We will use the following notation. Let $a \in \RR^n$, $|a|$ indicates its euclidean norm and $(a)_l$ its $l$-th component, while for $A$ Borel set $A \subset \RR^n$, $|A|$ indicates its Lebesgue measure. The symbols $\prec$ and $\preceq$ indicates the partial ordering with respect to the cone $\RR^m_{>0}$ and $\RR^m_{\geq 0}$ respectively.

\subsection{Pareto optimality and diversity}

When dealing with a vector-valued objective function $g: \RR^d \to \RR^m$
\begin{equation}
g(x)  = \left (g_1(x), \cdots, g_m(x) \right)
\end{equation}
with $m\geq 2$, the interpretation of the minimization problem \eqref{eq:mop} is not unique, as the image space $\RR^m$ is not fully ordered. We consider the notions of strong and weak Edgeworth-Pareto optimality which rely on the natural, component-wise, partial ordering on $\RR^m$ \cite{jahn2004vector}.

\begin{definition}[Edgeworth-Pareto optimality.] 
A point $\bar x \in \mathbb{R}^d$ is (strong) Edgeworth-Pareto (EP) optimal, or simply optimal, if $g(\bar x)$ is a minimal element of the image set $g(\mathbb{R}^d)$ with respect to the natural partial ordering, that is if there is no $x \in \mathbb{R}^d$ such that
\[ g_i(x) \leq g_i(\bar x) \;\; \textup{for all}\;\; i=1, \dots, m\,, \quad g(x) \neq g(\bar x)\,.\]
Alike, $\bar x$ is weakly EP optimal, if there is no $x \in \mathbb{R}^d$ such that  
\[ g_i(x) < g_i(\bar x)\; \; \textup{for all}\;\; i=1, \dots, m\,.
\]

\noindent The set 
$
F_x = \{ \bar x \in \RR^d \, |\, \bar x \; \textup{is EP optimal} \}
$
constitutes the set of optimal EP points, while
\[
F = \{ g(\bar x) \in \RR^m \, |\, \bar x \; \textup{is EP optimal} \}
\]
is  the Pareto front.
\end{definition}

The multi-objective optimization problem \eqref{eq:mop} consists of finding the set of EP optimal points. Unlike single-objective problems, the set is typically uncountable and the optimization task  involves finding a finite subset of optimal points. Those should ideally cover $F$ and  the concept of \textit{diversity} is introduced to distinguish between two approximations \cite{deb2001multi}. Intuitively, if  points on the Pareto front are more distanced, the diversity is higher.  In view of the minimization problem, having a diverse approximation is desirable as it provides at the same cost  a broader variety of possible solutions. 

The most diverse approximation possible is possibly given by a set of point which is uniformly distributed over the Pareto front.  Quantifying the diversity of an optimal set is of paramount importance both, to assess the performance of optimization methods and to design them. Indeed, oftentimes the heuristic of a specific method is constructed to specifically minimize, or maximize, a specific measure \cite{coello2020survey}. 
Without knowledge of the exact Pareto front, popular diversity measures are given by \textit{hypervolume contribution} \cite{zitzler1998multi}, \textit{crowding distance} \cite{deb2002nsga2} and, recently, by the \textit{Riesz s-energy} \cite{coello2021overview, vega2021towards}. Our proposed algorithm will aim to minimize the latter (or similar energy-based measures) as it can be embedded in a mean-field framework. The exact definitions are introduced later.

To sum up, the multi-objective optimization problem we consider is a two-objective task itself, as one needs to find a set of points which are both EP optimal and optimize a suitable diversity measure.

\subsection{Scalarization strategy}

A popular way to approach \eqref{eq:mop} is to use a scalarization strategy \cite{jahn2004vector,book2005mop} which reduces the multi-objective problem to a (finite) number of single-objective sub-problems. Among the possible scalarization strategies, we consider the approximation sub-problems with weighted Chebyschev semi-norms \cite{jahn2004vector} where the single objectives are given by
\be
G (x,w) := \max_{k\in \{ 1, \dots, m\}}\, w_k  \,|g_k(x)|\,. \notag
\ee
and are parametrized by a vector of weights $w$ which belongs to the unitary, or probability, simplex 
\be
\Omega :=\left \{ w \in \mathbb{R}^m_{\geq0}\; | \;\sum_{i=1}^m w_i = 1 \right\}\,.
\notag
\ee
For each $w \in \Omega$ the subproblems then read
\be
\min_{x \in \mathbb{R}^d} G(x, w).
\label{eq:sub}
\ee

The link  between the scalarized problems and the original multi-objective problem is given by the following result.
\begin{theorem}[{\cite[Corollaries 5.25, 11.21]{jahn2004vector}}]
Assume $g$ is component-wise positive.
\begin{enumerate}
\item[a)] A point $\bar x$ is weakly EP optimal if and only if $\bar x$ is a solution to \eqref{eq:sub} for some $w \in \Omega$.
\item[b)] Assume all sub-problems \eqref{eq:sub} attains an unique minimum. Then, $\bar x$ is EP optimal if and only if $\bar x$ is the solution to \eqref{eq:sub} for some $w \in \Omega$.
\end{enumerate}
\label{t:pareto}
\end{theorem}

\cref{t:pareto} shows the strength of the Chebyschev scalarization strategy which  allows to find all the weakly EP optimal points, contrary to other strategies like linear scalarization \cite{jahn2004vector}. We remark that the proposed algorithm can also be applied to solve any other scalarized problems of the form \eqref{eq:sub} where the parameters are take from the unitary simplex $\Omega$.

Even though solving $N$ sub-problems with corresponding weights vectors $\{ W^i\}_{i=1}^N \subset \Omega$ ensures to find $N$ optimal points, we note that there is no guarantee to obtain a diverse approximation. Therefore, scalarization targets only one of the two objectives of the problem, without addressing the diversity of the solution. In the following, we introduce an algorithm where the parameters $W^i$ are dynamically changed during the computation to obtain a set of EP points which is also diverse.

\section{Adaptive multi-objective consensus based optimization}
\label{sec:3}
We propose a dynamics where $N\in \mathbb{N}$ particles interact with each other to solve $N$ scalar sub-problems given in the form \eqref{eq:sub}. We introduce  the dynamics as a continuous-in-time process and leave the definition of the actual discrete optimization method to \cref{sec:5}. 

At a time $t\geq0$, every particle is described by its position $X_t^i \in \RR^d$ and its vector of weights $W_t^i \in \Omega$ which determines the optimization sub-problem the particle aims to solve. As a result, particles are described by $N$ tuples
\begin{align*}
(X_t^i, W_t^i)  \quad \textup{for}\quad i= 1, \dots, N\, \quad \textup{for all} \quad  t>0.
\end{align*} 
in the augmented space $\RR^d \times \Omega$.

The initial configuration  is generated by sampling the positions $X^i_0$ from a common distribution $\rho_0 \in \mathcal{P}(\RR^d)$ and by taking uniformly distributed weights vectors $W^i_0$ over $\Omega$. 
The dynamics is prescribed to solve the multi-objective optimization task. We recall that \eqref{eq:mop} not only requires to find optimal points, but also points that are diverse, that is, well-distributed over the Pareto front. To this end, the optimization process is made of two mechanisms which address these two objectives separately. 

\subsection{A consensus based particle dynamics in the search space}

The first mechanism prescribes the  update of the position $\{X_t^i\}_{i=1}^N$, such that they converge towards EP optimal points. As in \cite{borghi2022multi}, this is done by introducing a CBO-type dynamics between the particles. To illustrate the CBO update rule, let us consider for the moment a fixed single-objective sub-problem parametrized by $w \in \Omega$. Similar to Particle-Swarm Optimization methods, in CBO dynamics at time $t>0$, the particles instantaneously move towards an attracting point $Y_t^\alpha$ which is given by a weighted average of their position:
\be
Y^\alpha_t(w) = \frac{ \sum_{j=1}^N   X^j_t\, \exp\left(-\alpha G(X^j_t, w)\right)}{\sum_{j=1}^N \exp\left(-\alpha G(X^j_t, w)\right)}.
\label{eq:Ya}
\ee

Due to the coefficients used in \eqref{eq:Ya}, if $\alpha\gg1$, $Y^\alpha_t(w)$ is closer to the particles with low values of the objective function $G(\cdot, w)$ and, in the limiting case, it holds
\be \notag
Y^\alpha_t(w) \longrightarrow \underset{X^j_t,\,j=1,\dots, N}{\textup{argmin}} G(X_t^j, w)\quad \textup{as} \quad \alpha \to \infty\,
\ee
if the above minimum uniquely exists. This promotes the concentration of the particles in areas of the search space where the objective function $G(\cdot,w)$ attains low values and hence, more likely, a global minimum. We remark that the exponential coefficients correspond to the Gibbs distribution associated with the objective function and, moreover, that this choice is justified by the Laplace principle \cite{Dembo2010}. The later is  an essential result to study the convergence of CBO methods \cite{fornasier2021consensusbased} and it states that  for any absolutely continuous probability density $ \rho \in \mathcal{P}(\RR^d)$ we have
\be\notag
\lim_{\alpha \to \infty}\left( - \frac 1 \alpha \log \left( \int  e^{-\alpha G(x, w)} d\rho (x) \right) \right) = \inf_{x \in \text{supp}(\rho)} G(x, w)\,.
\label{eq:laplace}
\ee

Since in the multi-objective optimization dynamics each particle addresses a different sub-problem, each of them moves towards a different attracting point given by $Y^\alpha_t(W^i_t)$. 
The drift strength is given by $\lambda >0$, while another parameter $\sigma>0$ determines the strength of an additional stochastic component. 

The time evolution of the particles positions is determined by a system of SDE
\be
dX^{i}_{t} = \lambda\left(Y_t^\alpha(W^i_t) - X_t^i \right)dt +  \sigma D^i_t dB_t^i \quad \textup{for all} \quad  i = 1, \dots, N \, ,
\label{eq:sdex}
\ee 
where $B_t^i$ are $d$-dimensional independent Brownian processes. The matrices $D_t^i$ characterize the random exploration process  which might be isotropic \cite{pinnau2017consensus}
\be
D_{t,\text{iso}}^i =  | X_t^i - Y^\alpha_t(W^i_t)|\, I_d\,,
\label{eq:iso}
\ee
$I_d$ being the $d$-dimensional identity matrix, or anisotropic \cite{carrillo2019consensus}
\be
D_{t,\text{aniso}}^i =  \text{diag} \left( (X_t^i - Y^\alpha_t(W^i_t))_1 , \dots ,(X_t^i - Y^\alpha_t(W^i_t))_d  \right)\,.
\label{eq:aniso}
\ee
Both explorations depend on the distance between $X^i_t$ and the correspondent attracting point making the stochastic component larger if the particle is far from $Y_t^\alpha(W^i_t)$. The difference lays on the direction of the random component: while in the isotropic exploration all dimensions are equally explored, the anisotropic one explores each dimension with a different magnitude.

The expected outcome of the position update rule \eqref{eq:sdex} is that every particle will find a minimizer of a sub-problem and hence, by \cref{t:pareto}, a weak EP optimal point.  We have already mentioned that if the weights vectors are fixed to the initial, uniform, distribution $\{W_0^i \}_{i=1}^N$, that is
\[ \frac{d W_t^i}{dt} = 0 \quad \textup{for all} \quad  i = 1, \dots, N \, , \] 
there is no guarantee to obtain equidistant points on the front. Since it is impossible to determine beforehand the optimal distribution on $\Omega$, we propose a heuristic strategy which updates the vector weights  promoting diversity. 

\subsection{Uniform approximation of the Pareto front}

A popular diversity metric in multi-objective optimization is the \textit{hypervolume} contribution metric \cite{zitzler1998multi}, which has the drawbacks of being computationally expensive \cite{fonseca2009hyp} and, by definition, dependent on an estimate of $g$. Motivated by this and by the objective of designing algorithms which perform well for any shape of the Pareto front \cite{vega2021towards}, new energy-based diversity measures have recently gained popularity \cite{coello2020survey,coello2021overview}. Such measures quantify the diversity of a given  empirical distribution  $\rho^N \in \mathcal{P}(\RR^d)$ by considering the pairwise interaction given by a two-body potential $U: \RR^m \to (-\infty, \infty]$ on the image space
\be
\mathcal{U}[ g\# \rho^N ] :=  \iint U\left (g(x) - g(y) \right)\, d\rho^N(y)\,  d\rho^N(x)\,,
\label{eq:U}
\ee
$g \# \rho^N$ being the push-forward measure of $\rho^N$.

The problem of finding well-spread points over the Pareto front is then equivalent to finding a configuration which is minimal with respect to the given energy
 $\mathcal{U}$ where we recall that $F_x$ is the set of EP optimal points:
 
$$ \underset{\rho^N \in \mathcal{P}(F_x)}{\textup{min}}\; \mathcal{U}\, [ g \# \rho^N] \to \min.$$
A distribution $\nu^N$ is called diverse, if and only if 
\begin{equation*}
\nu^N \in \underset{\rho^N \in \mathcal{P}(F_x)}{\textup{argmin}}\; \mathcal{U}\, [ g \# \rho^N]\,.
\end{equation*}

Any energy $ \mathcal{U}$ describing short range repulsion between particles, like Monge energy or repulsive-attractive power-law energy, is in principle a  candidate to be a diversity measure. The Riesz $s$-energy given by
\be
U_R(z) = \frac{1}{|z|^s} \quad \textup{with} \quad s = m-1
\label{eq:riesz}
\ee
is a popular choice \cite{coello2021overview} due to its theoretically guarantees of being a good measure of the uniformity of points. Indeed, if $F$ is a $(m-1)$-dimensional manifold, the minimal energy configuration $\nu^N$ converges to the uniform Hausdorff distribution over $F$ as $N\to \infty$. We refer to \cite{hardin2005minimal} for the precise statements of the result and more details.  
Inspired by the electrostatic potential between charged particles, the authors in \cite{braun2015preference} used a Newtonian potential which is also empirically proven to be a suitable diversity measure \cite{braun2015preference,Braun2018thesis}. See \cite{coello2021overview} for a numerical comparison between two-body potentials as diversity measures in evolutionary algorithms. 
We will also compare different energies in \cref{sec:5} and consider $U\in \mathcal{C}^1(\RR^m \setminus \{ 0\})$ to be any of the above.
Exact computation of minimal energy configurations of a system of $N$ particles is a well-studied problem  as it is connected to e.g. crystallization phenomenon \cite{blanc2015crystal}. We note that, in our settings, the configuration $\rho^N$ is additionally mapped to the image space in \eqref{eq:U}, making the task even harder. Therefore, we propose an heuristic strategy that  is expected to find only suboptimal configurations.

To promote diversity, we let the particles  follow a vector field associated with $\mathcal{U}$. The movement will be only in parameter space  $\{ W_t^i\}_{i=1}^N$ in order not interfere with the CBO optimization dynamics acting on the positions $\{X_t^i\}_{i=1}^N$. 
Intuitively, if two particles are close to each other in the image space $g(\RR^d)$, their weights vectors are pulled apart. This resemble a short range repulsion of $U$. To ensure $W_t^i$ remains in the unitary simplex $\Omega$, a projection to the tangent cone $T(W_t^i,\Omega)$ 
\begin{equation*}
P_{W_t^i} (h):= P_{T(W_t^i, \Omega)} (h) = \left\{ z \in T(W_t^i,\Omega) \,:\, |z - h| = \inf_{\xi \in T(W_t^i,\Omega)} | \xi - W_t^i|\right \}  
\end{equation*}
for all $h \in \RR^m$ is required, see also \cite{carrillo2016nonlocal} for more details. A parameter $\tau\geq0$ determines the time scale of the weights adaptation process with respect to the CBO dynamics \eqref{eq:iterx}. The process can be turned off for  $\tau =0$.

In case of bi-objective problems, where $m=2$, we therefore obtain a Vlasov-type dynamics
\begin{equation}
\frac{dW^{i}_{t}}{dt} =  -  P_{T(W^i_t,\Omega)}\left ( -\frac{\tau}N \sum_{j=1}^N\nabla U\left(g(X_t^i) - g(X_t^j)\right) \right) \quad \textup{for all} \quad  i = 1, \dots, N \, ,
\label{eq:sdew}
\end{equation}
which is well-defined as the parameters space is embedded in the image space $\RR^m$. If $U$ has singularity in $0$, we set $\nabla U(0) = 0$.  We note that the additional minus sign in \eqref{eq:sdew}, is due to explicit form of the relation determined by \cref{t:pareto} between the Pareto front and $\Omega$. This will become clear in the next section, as this choice gives a gradient flow structure to the parameters dynamics.

For $m> 2$, the relation between a weight vector $w \in \Omega$ and the correspondent (weakly) EP optimal point is more involved. Nevertheless, we  prescribe a suitable heuristic dynamics as follows: let $U$ be given as
\begin{equation*}
U (z) = r(|z|) \quad \textup{for some}\quad r \in \mathcal{C}^1(\RR_{\geq0}),
\end{equation*}
then the parameters dynamics reads
\be
W^{i}_{t} =   P_{T(W^i_t,\Omega)} \left ( - \frac \tau {N}\sum_{j=1}^N \frac{W_t^i - W_t^j}{|W_t^i - W_t^j|}\; r' \left(|g(X_k^i) - g(X_k^j)|\right)\right)
\label{eq:iterw}
\ee
for all $i = 1, \dots, N$. The term $r'(\cdot)$ determines the strength and the sign of the interaction , while $(W_t^i - W_t^j)/|W_t^i - W_t^j|$ the direction of movement. As before, the projection step is needed due to the boundedness of $\Omega$. Even though \eqref{eq:iterw} can also be used when $m=2$, we will consider in the next section \eqref{eq:sdew} only.

Up to our knowledge, energy-based diversity metrics have only been used a selection criterion between candidate approximation of the Pareto front \cite{vega2021towards,coello2020reference}, and this is the first time the vector field associated to $\mathcal{U}$ is used to guide the particle dynamics in a metaheuristic multi-objective optimization method.

\section{Mean-field analysis of the particle dynamics}
\label{sec:4}
In this section, we give a statistical description of the optimization dynamics by presenting the corresponding mean-field model, which  allows us to analyze the convergence of the method towards a solution to the multi-objective optimization problem. We restrict ourselves to the case where $m=2$ and the dynamics in $\Omega$ is given by \eqref{eq:sdew}. The particle dynamics is given by \eqref{eq:iterx} and \eqref{eq:iterw2}, respectively.  

Similar to \cite{pinnau2017consensus}, we formally derive the mean-field equation of the large system \eqref{eq:sdex}, \eqref{eq:sdew} by making the so-called \textit{propagation of chaos} assumption on the marginals. In particular, let $F^N(t)$ be the particles probability distribution over $(\RR^d \times \Omega)^N$ at a time $t \geq0$. We assume that $F^N(t) \approx f(t)^{\otimes N}$ that is, that the particles $(X_t^i, W_t^i), i=1, \dots,N$ are independently distributed according to $f(t) \in \mathcal{P}(\RR^d \times \Omega)$ for some large $N\gg 1$.

In the following, we indicate with $\rho(t) \in \mathcal{P}(\RR^d)$ the first marginal of $f(t)$ and with $\mu(t) \in \mathcal{P}(\Omega)$ the second marginal on the parameters space $\Omega$. As a consequence of the propagation of chaos assumption, we obtain that 
\[
Y_t^\alpha(W^i_t) =\frac{\frac1N\sum_{i=1}^N X_t^i e^{-\alpha G(X_t^i,W^i_t)}}{\frac1N\sum_{i=1}^N e^{-\alpha G(X_t^i,W^i_t)}} \; \approx\;  \frac{\int x e^{-\alpha G(x, W^i_t)} d\rho(t)}{\int e^{-\alpha G(x, W^i_t)}d\rho(t)} =: y^\alpha (\rho(t), W^i_t) 
\]
and that
\[
\sum_{i=1}^N \nabla U \left( g(X_t^i) - g(X_t^j) \right) \approx \int \nabla U \left(g(X_t^i) - g(x) \right) d\rho(t)\,.
\]
The dynamics \eqref{eq:sdex}, \eqref{eq:sdew} is  now independent on the index $i$ and we obtain the  process $(X_t, W_t), t~>~0$ as
\be
\begin{cases}
dX_t &=\lambda (y_t^\alpha(\rho(t), W_t) - X_t)dt + D(\rho(t),W_t) dB_t \\
dW_t&=  \tau   P_{W_t} \left(\int \nabla U \left( g(X_t) - g(x)\right) d\rho(t)\right)\, dt
\end{cases}
\label{eq:mono}
\ee
where, $D(\rho_t,W_t)$ is defined consistently with \eqref{eq:iso} and \eqref{eq:aniso}. Process \eqref{eq:mono} is reformulated as
\begin{multline}
\frac\partial{\partial t} f(t,x,w) = - \lambda\nabla_x\cdot \Big( ( y^\alpha(\rho(t),w) - x) f(t,x,w) \Big) 
+ \frac{\sigma^2}2 \Delta_x \big( D(\rho(t),w) f(t,x,w)  \big) 
\\
- \tau \nabla_w \cdot \left( P_w\left(\int \nabla U\left(g(x) - g(y)\right) d\rho(t,y) \right) f(t,x,w) \right)\,,
\label{eq:mf}
\end{multline}
with initial conditions $f(0,x,w) = \rho_0 \otimes \mu_0$,  $\mu_0$ being the uniform distribution over the unit simplex $\Omega$, $\mu_0 = \textup{Unif}(\Omega)$.

The nonlinear partial differential equation \eqref{eq:mf} is a mean-field description of the microscopic dynamics generated by the optimization dynamics described in \cref{sec:3}. We note that the rigorous mean-field limit for single-objective CBO dynamics, which are similar to \eqref{eq:sdex}, was proven in \cite{huang2021meanfield}.
Following previous works, see e.g. \cite{carrillo2018analytical, fornasier2021consensusbased}, we consider such an approximation and mathematically analyze the proposed optimization method by studying a solution $f$ to \eqref{eq:mf}.

\subsection{Convergence to the Pareto front}
\label{s:4.2}
In the following, we assume $f \in \mathcal{C}\left([0,\infty), \mathcal{P}_2(\RR^d \times \Omega)\right)$ to be a solution to \eqref{eq:mf} with initial data given by $\rho_0 \in \mathcal{P}_2(\RR^d)$, $\mu_0 \in \mathcal{P}(\Omega)$.  We assess the performance of a multi-objective algorithm by the average distance form the Pareto front $F$ and we use the  Generational Distance ($GD$) \cite{van1998evolutionary} given by 
\be
GD[\rho(t)] = \left( \int \textup{dist}(g(x),F)^2 d\rho(t,x) \right)^{\frac{1}2}
\label{eq:GD}
\ee

where $\rho(t)$ is the first marginal of $f(t)$.
In the following, we state conditions such that  $DG[\rho(t)]$ decays up to a given accuracy $\ve>0$. 

\begin{assumption}[Uniqueness]
Every sub-problem \eqref{eq:sub} $w\in \Omega$ admits a unique solution $\bar x (w) \in F$. Moreover $\bar x \in \mathcal{C}^1 (\Omega,  \RR^d)$.
\label{a:1}
\end{assumption}

The uniqueness requirement is common in the analysis of CBO methods \cite{fornasier2021consensusbased}. This is due to difficulty to control  the attractive term $y^\alpha$, whenever there are two or more  minimizers. For example, assume $\nu$ is a measure concentrated in two different global minimizes of a sub-problem $w$, $\nu  = \left(\delta_{\bar X_1} + \delta_{\bar X_2}\right)$: the attractive term could be located in the middle between them
\[
y^\alpha(\nu,w) = \frac12 \left (\bar X_1  + \bar X_2 \right)\,
\]
being obviously not(!) minimizer. The regularity assumption on $\bar x$ follows form the transport term in \eqref{eq:mf}  with respect to $w$, and may be dropped if the interaction in the weights space is not present.

The next assumption requires that all scalar objective functions $G(x,w), w \in \Omega$ have a common lower and upper bounds in a neighborhood of the minimizer. See also  \cite{rosasco2017geometry} and the references therein for more details on the following conditions.

\begin{assumption}[Stability at the minimizer]
In a neighborhood of their minimizer, $G(\cdot,w), w \in \Omega$ are $p$-conditioned and satisfy a growth condition: there exists a radius $R>0$, exponents $p>2, q>1$ and constants $c_1, c_2>0$ such that for all $w \in \Omega$
\[
c_1 | x - \bar x (w) |^p \leq G(x,w) - \min_{y\in \RR^d} G(y,w) \leq c_2| x - \bar x (w) |^{1/q} \quad \text{for all}\quad  x: \; |x - \bar x (w) | \leq R \,.
\] 
Moreover, outside such a  neighborhood, the function cannot be arbitrary close to the minimum: there exists $c_3>0$ such that for all $w \in \Omega$
\[
c_3 \leq G(x,w) - \min_{y\in \RR^d} G(y,w) \quad \text{for all}\quad  x:\; |x - \bar x (w) | \geq R \,.
\]
\label{a:2}
\end{assumption}

Finally, we assume the optimal EP points to be bounded. As in other CBO methods we also prescribe a condition on the  initial data $\rho_0$ and $\mu_0$.
\begin{assumption}[Boundedness and initial datum] The set $F$ of optimal points is contained by a bounded, open set $H \subset \RR^d, |H|>0$. The initial distribution $f_0$ is given by $f_0 = \rho_0 \otimes \mu_0$ with $\rho_0 = \textup{Unif}(H)$ and some $\mu_0\in \mathcal{P}(\Omega)$.
\label{a:3}
\end{assumption}

Assumptions \ref{a:1}--\ref{a:3} ensure that the results on the Laplace principle \cite{fornasier2021consensusbased} are applicable to all the different sub-problems \eqref{eq:sub} with uniform choice of  $\alpha$. Therefore, under such assumptions, it possible to prove the convergence of each  in the following sense. Let $\mathbb{E}_{f(t)}[|x -\bar x(w)|^2], f(t) \in \mathcal{P}(\RR^d\times \Omega)$ denote the average $\ell_2$-error
\be
\mathbb{E}_{f(t)}\left[|x -\bar x(w)|^2\right] = \int | x - \bar x (w)|^2\, df(t,x,w)\,
\label{eq:V}
\ee
then, it holds:
\begin{theorem}[ {\cite[Theorem 12]{fornasier2021consensusbased}, \cite[Theorem 2]{fornasier2022aniso}}]
Assume (\ref{a:1})--(\ref{a:3}), $\nabla U \in L^\infty(\RR^m)$ and let $f \in \mathcal{C}\left([0,\infty), \mathcal{P}_2(\RR^d \times \Omega)\right)$ be a solution to \eqref{eq:mf} with  initial datum $f_0$. Let $\kappa=d$ if isotropic diffusion \eqref{eq:iso} is used and $\kappa =1$ for anisotropic diffusion \eqref{eq:aniso}.

For any  accuracy $\ve$, $0<\ve<\mathbb{E}_{f_0}\left[|x -\bar x(w)|^2\right]$,  if
\be 
 \kappa \sigma^2 + C \frac{\tau}{\sqrt{\ve}} < \lambda\,, \quad \text{where} \quad  C := \sqrt{2} \| \nabla W \|_{L^\infty(\RR^m)} \|\nabla_w \bar x \|_{L^\infty(\Omega, \RR^d)}
\ee
and if $\alpha$ is sufficiently large, there exists a time $T>0$ such that 
\[
\mathbb{E}_{f(T)}\left[|x -\bar x(w)|^2\right] = \ve\,.
\]

Moreover,  for all $t \in [0,T]$ it holds 
\be
\mathbb{E}_{f(t)}\left[|x -\bar x(w)|^2\right]  \leq \mathbb{E}_{f_0}\left[|x -\bar x(w)|^2\right]  e^{ - \left( \lambda  - \kappa\sigma^2 - C \tau/\sqrt{\ve} \right) t  } \,.
\label{eq:Vdecay}
\ee

\label{t:CBO}
\end{theorem}

We remark that the choice of $\alpha$ depends on the estimates given in \cref{a:2} and in particular on the accuracy $\ve$. 

\begin{corollary}
Under the settings of \cref{t:CBO}, if $g$ is Lipschitz continuous it holds
\[
GD[\rho(T)] =\textup{Lip}(g) \sqrt{\ve}
\]
and, for all $t \in [0,T]$,
\[
GD[\rho(t)] \leq \textup{Lip}(g) \sqrt{\mathbb{E}_{f_0}\left[|x -\bar x(w)|^2\right]}
 \exp\left( - \frac{\lambda  - \kappa \sigma^2 -  C \tau/\sqrt{\ve} }{2} t    \right)\,,
\]
where $\rho(t) \in \mathcal{P}(\RR^d)$ is the first marginal of $f(t)$.
\label{c:1}
\end{corollary}
\begin{proof}
Since every sub-problem admits a unique solution (Assumption \ref{a:1}), by \cref{t:pareto} every solution $\bar x(w)$ is EP optimal and therefore its image $g(\bar x(w))$ belongs to the Pareto front $F$. Therefore 
\[
\textup{dist}(g(x),F) \leq  | g(x) - g(\bar x (w))| \leq \textup{Lip}(g) | x- \bar x(w)|
\]
from which follows that the generational distance $GD$ is bounded by the average $\ell_2$-error.
\end{proof}

\cref{t:CBO,c:1} show that CBO mechanism is able to successfully solve all sub-problems \eqref{eq:sub} simultaneously. In the next section, we will analyze the dynamics in the parameters space $\Omega$ to investigate the diversity of the computed solution.

\begin{remark}
In \cref{t:CBO}, $\tau$ needs to be taken of order $o(\sqrt{\ve})$ suggesting that the parameters should adapt at a much slower time scale with respect to the positions, in order not to interfere with the CBO dynamics. With no weights vectors interaction, $\tau=0$, the decay estimate \eqref{eq:Vdecay} is independent of $\ve$ and, in particular, the particles converge faster towards EP optimal points. 
\label{r:tau0}
\end{remark}

\subsection{Decay of diversity measure}
\label{s:4.3}

The aim of interaction \eqref{eq:iterw} is  improve the distribution of the parameters $\{W_t^i\}_{i=1}^N$ so that, in view of \cref{t:pareto}, the corresponding (weak) EP optimal points are well-distributed in the image space. Under suitable assumptions, such dynamics corresponds to a gradient flow on the unitary simplex $\Omega$.

For any sub-problem \eqref{eq:sub} parametrized by $w \in \Omega$, let $\bar x(w)$ be one of its global minima, which we assume exists. As we are interested in the relation between $w$ and its correspondent point on the Pareto front $g\left(\bar x(w)\right) \in F$, let us formally insert in the mean-field model \eqref{eq:mf} solutions of the form
\be
f(t,x,w) = \delta(x - \bar x(w) ) \mu(t,w)
\label{eq:f}
\ee
where $\mu(t) \in \mathcal{P}(\Omega)$. In ansatz \eqref{eq:f}, the location $x$ of the particle $(x,w)$ corresponds exactly to a solution $\bar x(w)$ to its sub-problem $w\in \Omega$. This is justified by the convergence result (\cref{t:CBO}) and by the fact that the positions dynamics takes place at a faster time scale then the parameters adaptation, see \cref{r:tau0}.

The reduced mean field equation in strong form is then given by  
\begin{multline}\notag
\frac\partial{\partial t} f(t,x,w) = - \lambda\nabla_x\cdot \Big( \left( y^\alpha(\rho(t),w) - \bar x(w)\right) f(t,x,w) \Big) 
+ \frac{\sigma^2}2 \Delta_x \big( D(\rho(t),w) f(t,x,w)  \big) 
\\
- \tau \nabla_w \cdot \left( P_w\left(\int \nabla U\left(g(\bar x(w)) - g(\bar x(v))\right) df(t,y,v) \right) f(t,x,w) \right)\, 
\label{eq:mfreduced}
\end{multline}
and, the marginal $\mu(t)$ over $\Omega$ fulfills 
\be
\frac{\partial}{\partial t} \mu(t,w) = - \tau  \nabla_w \cdot \left( P_w\left(\int \nabla U\left(\bar g(w) -\bar g(v)\right) d\mu(t,v) \right) \mu(t,w) \right)
\label{eq:red1}
\ee
where for simplicity we introduced $\bar g := g \circ \bar x$. 

\begin{assumption} The Pareto front $F$ is exactly the unitary simplex $\Omega$ and the potential energy $U$ is radially symmetric.
\label{a:4}
\end{assumption}

\begin{lemma} Under \cref{a:4}, for all $w,v \in \Omega$ it holds
\be
P_w\Big(  \nabla U\big(\bar g(w) - \bar g(v)\big)   \Big) = - P_w\big(\nabla U (w - v)  \big)\,.
\label{eq:P}
\ee
\label{l:1}
\end{lemma}
\begin{proof}
 We note that when $F = \Omega$, all weakly EP optimal points are also EP optimal and hence by \cref{t:pareto} $\bar g(w) = g(\bar x(w)) \in F$ for all $w \in \Omega$. Then, there exists $s \in [0,1]$ such that $\bar g (w) = (s, 1-s)$ for all $w$. By definition of the sub-problem \eqref{eq:sub} with $w = (w_1,w_2) = (w_1, 1-w_1)$,
 \[
 \bar g(w) = \min_{y \in F} \max \left\{ y_1w_1, y_2 w_2 \right\} = 
 \min_{s \in [0,1]} \max \{ s w_1, (1-s) (1-w_1)  \}\,.
 \]
At the minimizer, it must hold $s w_1 = (1-s)(1-w_1)$ and hence $s = (1-w_1)$. It follows that
 \be
\bar g (w) = A w, \quad \text{where} \quad A = \begin{pmatrix}
0 & 1\\
1 & 0 
\end{pmatrix}.
\label{eq:gbar}
 \ee
Since $U$ is radially symmetric it holds $\nabla U ( Aw - Av) = A \nabla U(w-v)$.

Finally, let us consider the basis $n_1 = (1,1)^\top, n_2 = (1,-1)^\top$ and a vector $u \in \RR^2, u = u_1 n_1 + u_2 n_2$. We note that $P_w$ always projects towards $n_2$. Together with the fact that $Au = u_1n_1 - u_2n_2$, this leads to
\[
P_w(Au) = P_w(-u_2 n_2) = P_w(-u)\, 
\]
and the identity \eqref{eq:P} follows.
\end{proof}

Thanks to \cref{l:1}, under \cref{a:4} equation \eqref{eq:red1} can be simplified to 
\be
\frac{\partial}{\partial t} \mu(t,w) =  - \tau  \nabla_w \cdot \left( P_w\left( - \int \nabla U\left(w -v )\right) d\mu(t,v) \right) \mu(t,w) \right)
\label{eq:red2}
\ee
and  initial conditions $\mu(0) = \mu_0$.
Equation \eqref{eq:red2} describes the continuum dynamics of particles which binary interact and that are confined to the set $\Omega$. Such aggregation model on bounded domains has been subject of several works, see for instance \cite{fatecau2017swarm,patacchini2022nonlocal}. Particularly relevant to the present work is \cite{carrillo2016nonlocal} where general prox-regular sets, like $\Omega$, are considered. 

\begin{theorem}[{\cite[Theorem 1.5 ]{carrillo2016nonlocal}}]
Assume $U \in \mathcal{C}^1(\RR^2)$ to be $\tilde\lambda$-geodetically convex on $\textup{Conv}(\Omega - \Omega)$ for some $\tilde\lambda \in \RR$. For any initial data $\mu_0 \in \mathcal{P}_2(\Omega)$ there exists a locally absolutely continuous curve $\mu(t) \in \mathcal{P}(\Omega), t>0,$ such that $\mu$ is a gradient flow with respect to $\mathcal{U}$. Also, $\mu$ is a weak measure solution to \eqref{eq:red2}. 

Furthermore,
\be
\frac{d}{dt}\,\mathcal{U}(\mu(t)) \leq - \int \left |  P_w\left( \nabla U \ast \mu(t) (w)      \right)\right|^2 \mu(t,w) \,,
\label{eq:decayU}
\ee
\label{t:flow}
where $*$ denotes the convolution operator.
\end{theorem}

Under \cref{a:4} and thanks to  relation \eqref{eq:gbar} between $\Omega$ and $F$, \cref{t:flow} states that the energy over the front is decreasing. We note that the flow may convergence to the stationary points of \eqref{eq:red1} that are not minimal configurations, as observed in \cite{fatecau2017swarm} for even simple domains. 

Clearly, without ansatz \eqref{eq:f}, there is no guarantee that the potential decreases along the evolution of the algorithm. 
Quite the opposite, by \cref{t:CBO} particles are expected to concentrate on the Pareto front leading to an increased potential $\mathcal{U}$. Nevertheless, by \cref{t:CBO} there exists a time $T>0$ where
\[
\int |  x- \bar x(w) |^2\, df(T,x,w) < \ve \quad \text{and hence} \quad x \approx \bar x(w)
\]
making ansatz \eqref{eq:f} valid. Therefore, we claim that the reduced model \eqref{eq:red2}  describes the dynamics for $t>T$. We will numerically investigate  two phases of the algorithm: the first one when  concentration over the Pareto happens, and the second  when the potential $\mathcal{U}$ decays leading the an improved diversity of the solution. 

%

\section{Numerical experiments}
\label{sec:5}

In this section, we numerically investigate the performance of the proposed method by testing it against several benchmark multi-objective problems. 

The adaptive multi-objective consensus based optimization (AM-CBO) algorithm is obtained from an Euler--Maruyama time-discretization of \eqref{eq:sdex} and \eqref{eq:sdew} (or \eqref{eq:iterw} if $m>2$).
Let $\Delta t>0$ be a fixed time-step. For $k=0,1,\dots$, the particles positions are iteratively updated according to
\begin{equation}
X^i_{k+1} = X_k^i + \lambda\left(Y_k^\alpha(W^i_k) - X_k^i \right)\Delta t  +  \sigma D^i_k \sqrt{\Delta t}B_k^i
\label{eq:iterx}
\end{equation}
for all $i = 1, \dots, N$ where $B_k^i$ are multivariate independent random vectors, $B_k^i \sim \mathcal{N}(0,I_d)$. The update rule \eqref{eq:iterx} is overparametrized and in CBO optimization schemes  typically  $\lambda=1$ is used.

Similar to the projected gradient flow scheme used in \cite{patacchini2022nonlocal}, we replace the instantaneous projection to the tangential space $T(w,\Omega)$ by the projection $\Pi_{\Omega}$ to $\Omega$,
\[
\Pi_{\Omega} (v)  = \left \{w \in \Omega\;:\: |v-w| = \inf_{\xi \in \Omega} |v - \xi|  \right \} \quad \textup{for} \quad v \in \RR^m 
\]
and discretize the dynamics in $\Omega$ as
\be
\begin{cases}
V^{i}_{k+1} &= W^i_k + \frac \tau {N}\sum_{j=1}^N \nabla U\left(g(X_k^i) - g(X_k^j)\right) \Delta t\\
W^{i}_{k+1} &= \Pi_{\Omega}\left( V^i_{k+1} \right) 
\end{cases}
\,,
\label{eq:iterw2}
\ee
for $m=2$, while for $m>2$ it reads
\be
\begin{cases}
V^{i}_{k+1} &= W^i_k - \frac \tau {N}\sum_{j=1}^N  \frac{W_k^i - W_k^j}{|W_k^i - W_k^j|}\; r' \left(|g(X_k^j) - g(X_k^i)|\right)\Delta t\\
W^{i}_{k+1} &= \Pi_{\Omega}\left( V^i_{k+1} \right) 
\label{eq:iterw_bis}
\end{cases}
\,.
\ee

The complete optimization method is described by Algorithm \ref{alg:1}. A remark on the  computational complexity follows.

\begin{algorithm}
\caption{AM-CBO} 
\label{alg:1}
\begin{algorithmic}
\STATE{Set parameters: $\alpha, \lambda, \sigma, \tau, \Delta t$}
\STATE{Initialize the positions: $X^i_0 \sim \rho_0\,, i=1, \dots,N$}
\STATE{Initialize the weights vectors $\{ W_0^i\}_{i=1}^N $ uniformly in $\Omega$}
\STATE{$k\gets 0$}
\WHILE{stopping criterion is NOT satisfied}
	\STATE{Compute $g(X^i_k)\,, i=1, \dots,N$ }
	\FOR{$i=1, \dots, N$}
		\STATE{compute $Y_k^\alpha(W_k^i)$ according to \eqref{eq:Ya}}
		\STATE{sample $B^{i}_{k}$ from $\mathcal{N}(0, I_d)$}
		\STATE{update  $X^i_{k+1}$ according to \eqref{eq:iterx}}
		\STATE{update  $W^i_{k+1}$ according to \eqref{eq:iterw2} (or \eqref{eq:iterw_bis}})
	\ENDFOR
	\STATE{$k \gets k+1 $ }
\ENDWHILE
\RETURN $\{ X_{k}^i \}_{i=1}^N$
\end{algorithmic}
\end{algorithm}

For the sake of reproducible research, in the GitHub repository 
\url{https://github.com/borghig/AM-CBO}
an implementation in MATLAB code of the proposed algorithm is made available.

\begin{remark}
Even though in every iteration the objective function $g$ is evaluated only $N$ times, the overall computational complexity is $\mathcal{O}(N^2)$ because the computation of $Y^i_k(w)$ requires $\mathcal{O}(N)$ computations, as well as the parameters update \eqref{eq:iterw} which is particularly costly. 

One can reduce the computation complexity by considering only a random subset $I_k^M \subset \{1, \dots, N \}$ of $M \ll N$ particles when computing \eqref{eq:Ya} and \eqref{eq:iterw2}, by substituting
\[
\frac1N \sum_{j=1}^N (\cdot)^j \quad \textup{with} \quad  \frac1{M} \sum_{j \in I^M_k} (\cdot)^j\,,
\]
whenever a sum over the different particles is performed.  Inspired by Monte-Carlo particle simulations \cite{AlPa, JLJ}, this mini-random batch technique allows to lower the complexity to $\mathcal{O}(NM)$. We also note that that Fast Multipole Methods (FMM) \cite{greengard1987fast} may additionally be  used to speed up the computation of the potential field, Then,  the  computational complexity of \eqref{eq:iterw2}, \eqref{eq:iterw_bis} is further reduced.
\end{remark}

\subsection{Performance metrics}

Denote by $\{X_k^i\}_{i=1}^N$ the set of particle positions  at the $k$-th algorithm iteration and their  empirical distribution by  $\rho_k^N \in \mathcal{P}(\RR^d)$.
We employ three different energies, the Riesz $s$-energy \eqref{eq:riesz}, Newtonian and the Morse potentials, both to measure the solutions diversity and to determine the dynamics of the vector weights. The Newtonian binary potential is given by
\be
U_N (z) = 
\begin{cases}
\log(|z|)      & \textup{if} \;\; m = 2 \\
 |z|^{2-m}    & \textup{if}\;\;   m >2
\end{cases}
\;, 
\label{eq:newt}
\ee
while the Morse potential is given
\be
U_{M}(z) = e^{-C|z|} \quad \textup{with}\quad C>0. 
\label{eq:morse}
\ee
All  considered potentials describe short-range repulsion between the particles. While the Morse potential is $\tilde \lambda$-geodetically convex, the Newtonian and Riesz repulsion are not. Since we will also employ the corresponding energies $\mathcal{U}_R$, $\mathcal{U}_N$, $\mathcal{U}_M$ to define the interaction between  parameters, the constant $C$ can be considered as an algorithm parameter when the Morse repulsion is used.

To show the validity of the energy-based diversity metrics, we additional consider the hypervolume contribution metric $\mathcal{S}$ \cite{zitzler1998multi}. Let $g^* \in \RR^m$ be a maximal element with respect to the natural partial ordering 
\[
y_i \prec g^*_j \quad \textup{for all}\quad y \in F\,,
\]
the hypervolume measure is given by the Lebesque measure of the set of points between the computed solution and the maximal point $g^*$, that is
\be
\mathcal{S} [\rho_k^N] = \left | \bigcup_{i=1}^N \left \{y \in \RR^m \; | \;  g(X^i_k) \prec y \prec g^* \right\} \right|.
\label{eq:hyp}
\ee
Maximizing $\mathcal{S}$ has been shown to lead to a diverse approximation of the Pareto front \cite{emmerich2005emo}.

In \cref{s:4.2},  the convergence of the mean-field dynamics towards the Pareto front is shown by studying the evolution of the Generation Distance $GD$  \eqref{eq:GD}. In the experiments, we approximate this quantity by considering a reference approximation $\{ y^j\}_{j=1}^M$ of the front with  $M=100$ points $y^i \in F$, $i = 1, \dots, M$ for every test problem. More details on the reference solution are given  in \cref{app:1}.
For  simplicity, we indicate the numerical approximation of the Generational Distance again by $GD$, which is  defined by
\be
GD[\rho_k^N] = \left( \frac 1 N \sum_{i=1}^N \textup{dist}(g(X_k^i), F_M)^2
\right)^{\frac12}\,.
\label{eq:GDnum}
\ee

The  Inverted Generational Distance $IGD$ is also considered. It consists of the average distance between the points of the reference solution $\{ y^i\}_{j=1}^M$ and the computed front
\be
IGD[\rho_k^N] = \left( \frac 1 M \sum_{j=1}^M \textup{dist}(y^j, G_k)^2
\right)^{\frac12} \quad \textup{with}\quad G_k := \{g(X_k^i)\,|\,i = 1,\dots,N\}\,.
\label{eq:IGDnum}
\ee
Contrary to $GD$ which only measures the distance form the Pareto front, $IGD$ takes in account the diversity of the computed solution, too. Hence, $IGD$ is also a suitable indicator of the  optimality of the solution.

\subsection{Test problems}

Test problems with diverse Pareto front geometries are selected  to show the performance of the proposed method. In the Lamé problems \cite{emmerich2007lame} the parameter $\gamma$ controls the front curvature: we use $\gamma = 0.25, 1, 3$ to obtain convex, linear and concave fronts respectively.  We also consider the DO2DK \cite{branke2004finding} problems with $k=2,s=1$ and $k=4, s=2$. Here, the Pareto fronts have more complex geometries as they are not symmetric and, in one case,   discontinuous. 
All above problems are scalable to any dimension of the search space $d$ and in the image space $m$. For presentation purposes, we restrict ourselves to bi-objective optimization problems by setting $m=2$, but consider possibly large $d.$ In this case, the fronts analytical description are known, allowing us to obtain  reference solutions. The problems definitions are recalled in \cref{app:2} for completeness.

\newcommand\ww{0.96}

\begin{figure}

\centering
\includegraphics[trim = 0cm 0cm 0cm 0cm , clip, width=\ww\linewidth]{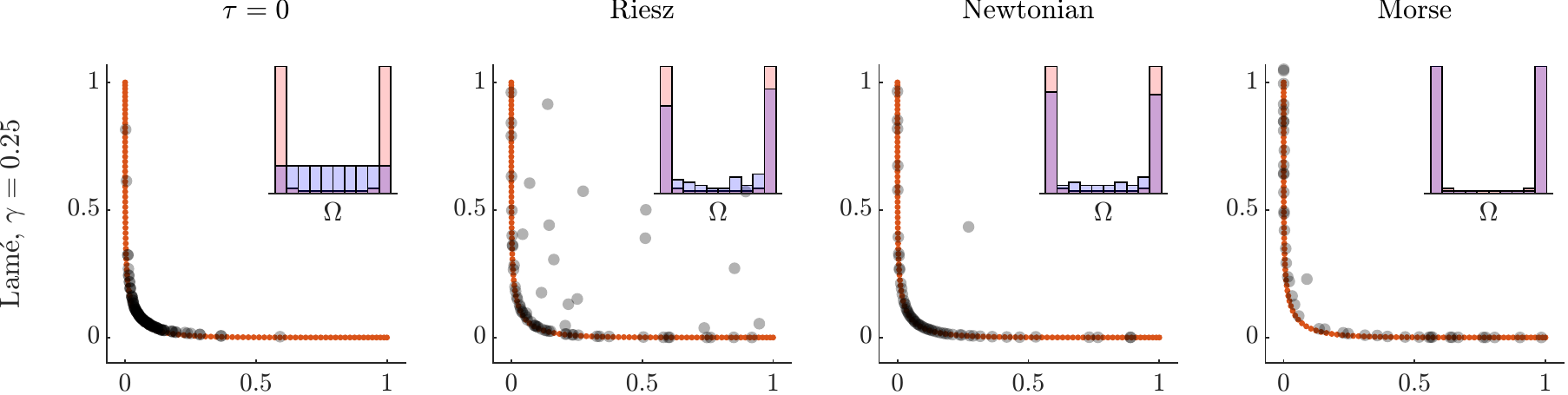}

\bigskip
\includegraphics[trim = 0cm 0cm 0cm 0.5cm, clip, width=\ww\linewidth]{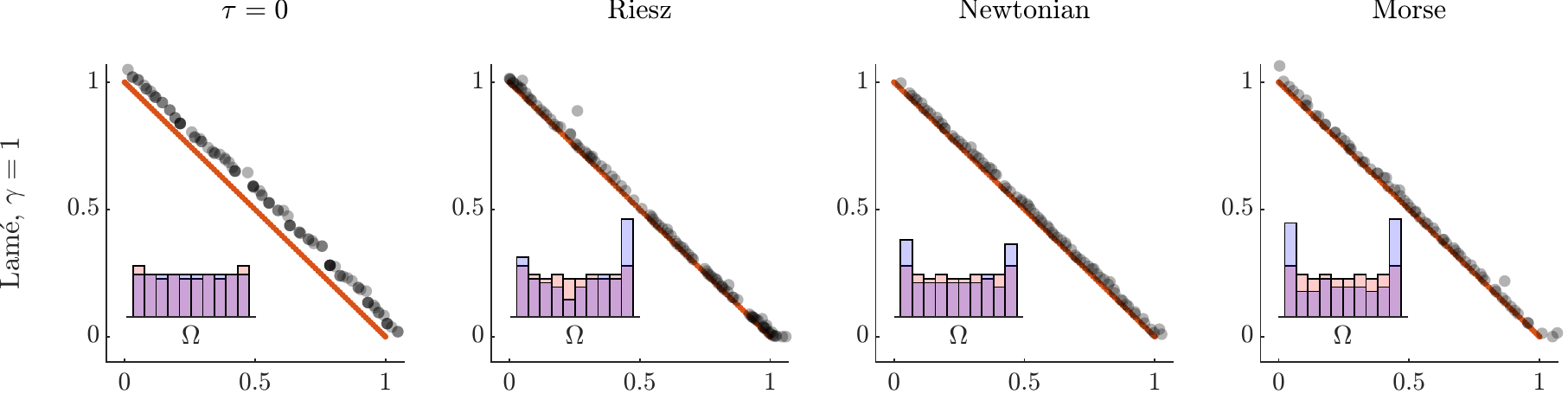}
\bigskip

\includegraphics[trim = 0cm 0cm 0cm 0.5cm , clip, width=\ww\linewidth]{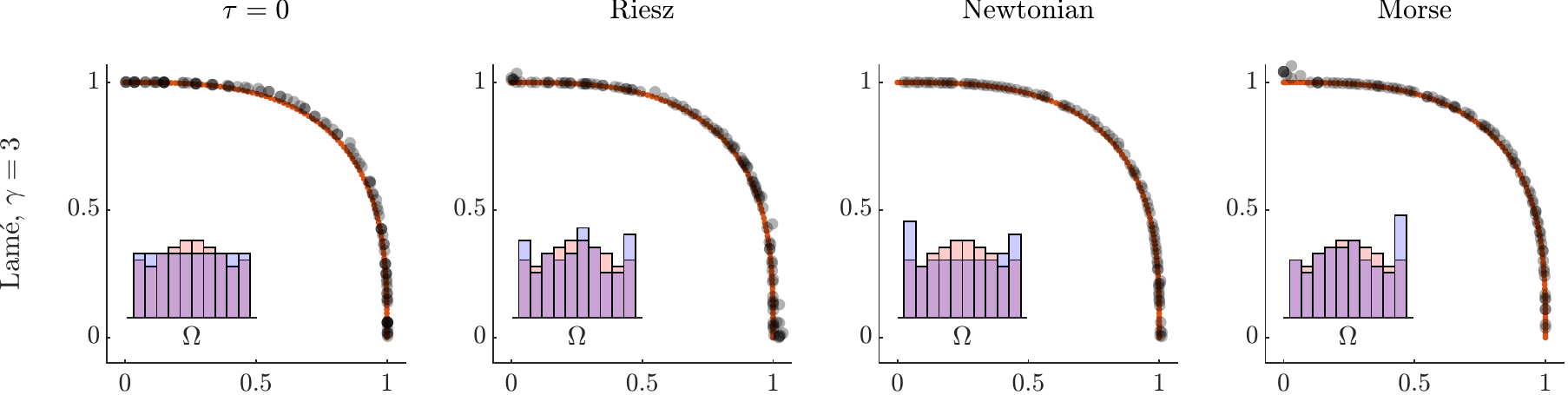}

\bigskip
\includegraphics[trim = 0cm 0cm 0cm 0.5cm , clip, width=\ww\linewidth]{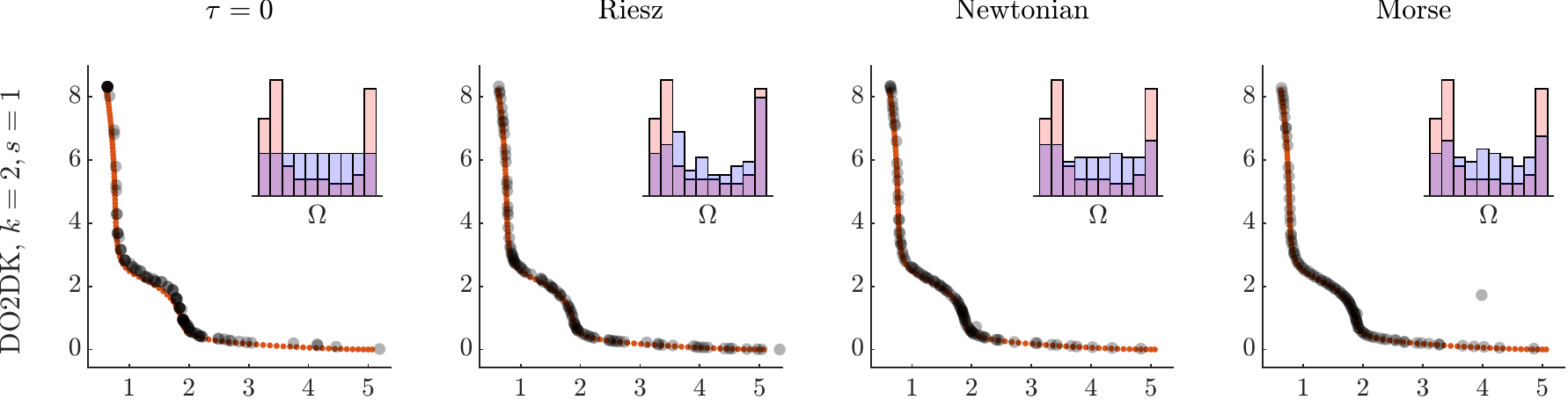}

\bigskip
\includegraphics[trim = 0cm 0cm 0cm 0.5cm , clip, width=\ww\linewidth]{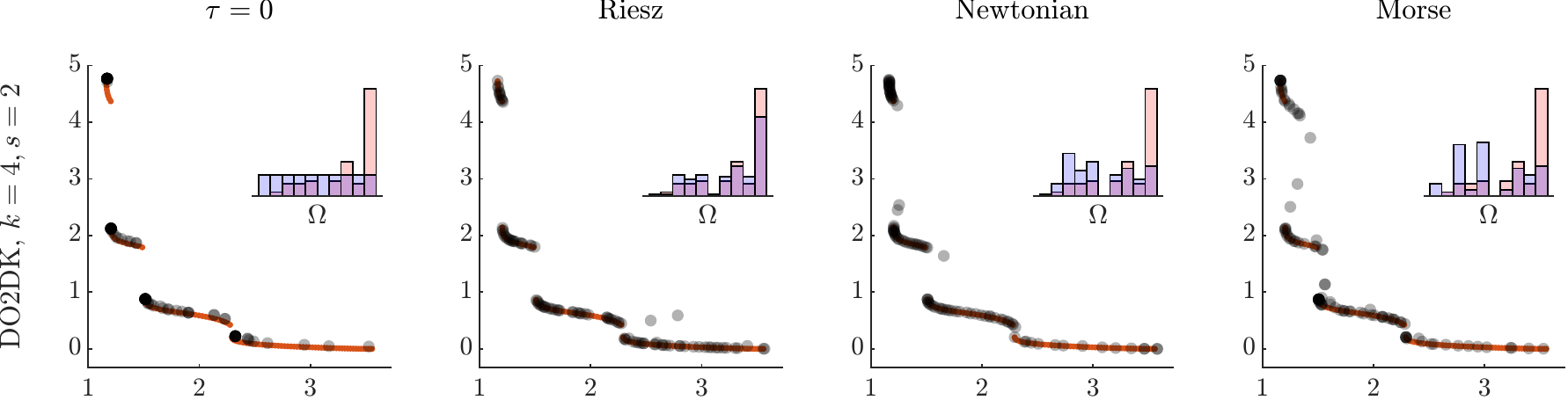}

\caption{In black, particles position in the image space after a single run. The reference
solution is displayed in red. Four different parameters interaction strategies are used: no
interaction, Riesz, Newtonian and Morse potential. Histograms show the final distribution
over $\Omega$ (blue) and the optimal one (red).}
\label{fig:1}
\end{figure}

In this section,  we use Algorithm \ref{alg:1} in four different scenarios

\begin{enumerate}
\item No parameters interaction $\tau=0$;  
\item Riesz potential \eqref{eq:riesz}, with $\tau=10^{-5}$\,;
\item  Newtonian potential \eqref{eq:newt}, with $\tau=10^{-3}$\,;
\item  Morse potential \eqref{eq:morse}, with $\tau=10^{-1}$, $C=20$;
\end{enumerate}
The first scenario clearly corresponds to the standard M-CBO approximation, while the others to different AM-CBO strategies.
To validate  model \eqref{eq:red2} and \cref{t:flow}, we update the parameters according to \eqref{eq:iterw2}.
The initial weights vectors $\{ W_0^i\}_{i=1}^N$ are taken (deterministically) uniformly distributed over $\Omega$, while the particle positions are uniformly sampled over $[0,1]^d$, $d=10$. We employ $N=100$ particles, which evolve for a maximum of $k_{\textup{max}} = 5000$ steps. The remaining parameters are set to $\lambda = 1, \sigma = 4, \alpha = 10^6$. This parameter choice consists of a compromise between the optimal parameters of each problem. Anisotropic diffusion \eqref{eq:aniso} is used and a projection step ensures the particle positions remain in the search space $[0,1]^d$, which is the same for all considered problems.

\cref{fig:1} shows the computed solutions, in the image-space, in the four different scenarios.  Regardless of the interaction on $\Omega$, the particles always converge towards $EP$ optimal points and hence to the Pareto front.  By definition of the Chebyshev sub-problems \eqref{eq:sub}, a uniform distribution in $\Omega$ leads to an uniform distribution of the particles over the front only when $F$ is linear (as in the Lamé problem $\gamma=1$). Indeed, \cref{fig:1} shows that the particles are well distributed even when there is not weights interaction ($\tau=0$). If the front geometry differs from this straight segment, the optimal parameters distribution on $\Omega$ differs form the uniform one. In particular, subsets of the Pareto front which are almost parallel to the axis are difficult to approximate without any  interaction in the parameter space, see for instance Lamé $\gamma = 0.25$ and the DO2DK problems in \cref{fig:1}. When using $\tau \neq 0$, the solutions improves as the particles are more distributed over the entire front.

\begin{table}
\centering
\begin{tabular}{|l|l|c|c|c|c|c|c|c|}
\hline
 Problem& Interaction & $GD$ & $\mathcal{U}_R$ & $\mathcal{U}_N$&$ \mathcal{U}_M$ & $\mathcal{S}$ & $IGD$ \\
\hline
Lam\'e 0.25& $\tau = 0$ & \cellcolor{g}2.33e-02 & 1.00e+10 & 2.41e+00 & 4.86e-01 &\cellcolor{g} 9.69e-01 & 1.31e-01 \\
\cline{2-8}
 	& Riesz & 8.74e+00 &\cellcolor{g} 5.65e+00 & -1.94e-01 & 9.62e-02 & 7.77e-01 & 4.06e-02 \\
\cline{2-8}
 	& Newtonian & 1.11e+01 & 8.23e+00 & -3.53e-01 & 1.14e-01 & 8.38e-01 & 4.25e-02 \\
\cline{2-8}
 	& Morse & 1.49e+01 & 1.81e+04 &\cellcolor{g} -1.36e+00 & \cellcolor{g}3.40e-02 & 7.45e-01 & \cellcolor{g}2.64e-02 \\
\hline
\hline
Lam\'e 1& $\tau = 0$ & \cellcolor{g}9.88e-02 & 9.60e+09 & 9.96e-01 & 1.26e-01 & 3.74e-01 & 8.28e-02 \\
\cline{2-8}
 	& Riesz & 1.63e-01 & \cellcolor{g}6.54e+00 & 8.57e-01 & 1.23e-01 & 4.59e-01 & \cellcolor{g}1.56e-02 \\
\cline{2-8}
 	&Newtonian & 9.81e-01 & 8.39e+00 & 5.77e-01 & 9.47e-02 & \cellcolor{g}4.62e-01 & 1.91e-02 \\
\cline{2-8}
 	&Morse & 6.83e-01 & 8.97e+05 & \cellcolor{g}4.41e-01 &\cellcolor{g} 7.95e-02 & 4.48e-01 & 1.78e-02 \\
\hline
\hline
Lam\'e 3& $\tau = 0$ &\cellcolor{g}1.93e-02 & 8.40e+09 & 9.56e-01 & 1.30e-01 & 8.45e-02 & 2.18e-02 \\
\cline{2-8}
 	& Riesz & 5.64e-02 & 7.06e+00 & 7.68e-01 & 1.14e-01 & 1.01e-01 & 1.32e-02 \\
\cline{2-8}
 	& Newtonian & 2.34e-01 & \cellcolor{g}6.33e+00 & 5.74e-01 & 9.10e-02 & \cellcolor{g}1.03e-01 & \cellcolor{g}1.11e-02 \\
\cline{2-8}
 	&Morse & 3.02e-01 & 9.57e+06 & \cellcolor{g}5.04e-01 & \cellcolor{g}7.84e-02 & 1.02e-01 & 1.29e-02 \\
\hline
\hline
DO2DK & $\tau = 0$ & 1.80e-01 & 1.00e+10 & -3.30e-01 & 6.94e-02 & 8.84e+01 & 2.82e-01 \\
\cline{2-8}
k=2,s=1  & Riesz & \cellcolor{g}5.03e-02 &\cellcolor{g} 1.58e+00 & \cellcolor{g}-6.04e-01 & 4.18e-02 & \cellcolor{g}8.94e+01 & 1.18e-01 \\
\cline{2-8}
 	& Newtonian & 6.48e-02 & 1.77e+01 & -5.98e-01 & 3.85e-02 & 8.94e+01 & \cellcolor{g}1.07e-01 \\
\cline{2-8}
 	& Morse & 9.59e-02 & 7.67e+08 & -5.64e-01 &\cellcolor{g}3.75e-02 & 8.94e+01 & 9.33e-02 \\
\hline
\hline
DO2DK &$\tau = 0$ & \cellcolor{g}6.60e-02 & 1.00e+10 & 2.69e+00 & 2.64e-01 & \cellcolor{g}8.66e+01 & 1.36e-01 \\
\cline{2-8}
k=4, s=2& Riesz & 8.95e-01 & \cellcolor{g}3.34e+00 & -1.27e-01 & \cellcolor{g}7.08e-02 & 8.40e+01 & \cellcolor{g}2.61e-02 \\
\cline{2-8}
 	& Newtonian & 1.50e+00 & 2.18e+01 &\cellcolor{g} -2.26e-01 & 7.52e-02 & 8.44e+01 & 3.61e-02 \\
\cline{2-8}
 	&Morse & 9.85e+00 & 1.94e+09 & -1.56e-01 & 9.09e-02 & 7.63e+01 & 3.45e-02 \\
\hline
\end{tabular}
\caption{Algorithm performance for the different settings and problems. Results are averaged over 25 runs.}
\label{table:results}
\end{table}

\cref{table:results} reports the performance metrics for all the problems. For most problems, the strategy $\tau=0$, with no interaction in $\Omega$ allows to reach lower values of $GD$. This is consistent with the analytical results \cref{t:CBO} and \cref{r:tau0}, which suggested that the additional dynamics may interfere with the CBO mechanism and, as a consequence, slow down the convergence towards optimal EP points. If  the diversity metrics $\mathcal{U}_R$, $\mathcal{U}_N$, $\mathcal{U}_M$ and $\mathcal{S}$ are considered, dynamics including interaction of parameters  allow to obtain more diverse solutions. Interestingly, using Morse binary potential in the interaction leads to a final lower Newtonian energy in some cases. We will  investigate the role of the potential choice and $\tau$ in the next section.

In \cref{fig:1} the  $IGD$ performance  shows that letting particles interact in parameter space improves the overall quality of the solution. While the improvement is more substantial in problems with complex Pareto fronts (see for instance Lamé $\gamma = 0.25$, or DO2DK $k=2$), we remark that the additional mechanism allows to obtain better solutions. This is even true, if  the parameter distribution is already optimal form the beginning (see Lamé $\gamma = 1$). We conjecture that this due to the additional stochasticity introduced by the potential. We will also study this aspect in the next subsection.

 \cref{fig:2a,fig:2b} show the time evolution of $GD$, $\mathcal{U}_R$, $\mathcal{U}_N$, $\mathcal{U}_M$ and $IGD$ for two of the considered test problems. As suggested by the analysis of the mean-field model, in particular \cref{t:CBO}, $GD$ exponentially decays up to a maximum accuracy within the first iterations of the algorithm. This is due to the $CBO$ dynamics driving the particles around EP optimal points. At the same time, the potential energies increase as the particles are concentrating towards the front in the image-space. Another consequence of \cref{t:CBO} is that assumption \eqref{eq:f} is fulfilled and consequently the gradient-flow description \eqref{eq:red2}  is valid. This is also observed in \cref{fig:2a,fig:2b} where the potentials start decreasing provided that relatively low $GD$ values are attained.

\begin{figure}
\centering
\begin{subfigure}{0.91\linewidth}
\centering
\includegraphics[trim = 0cm 0cm 0cm 0cm , clip, width=1\linewidth]{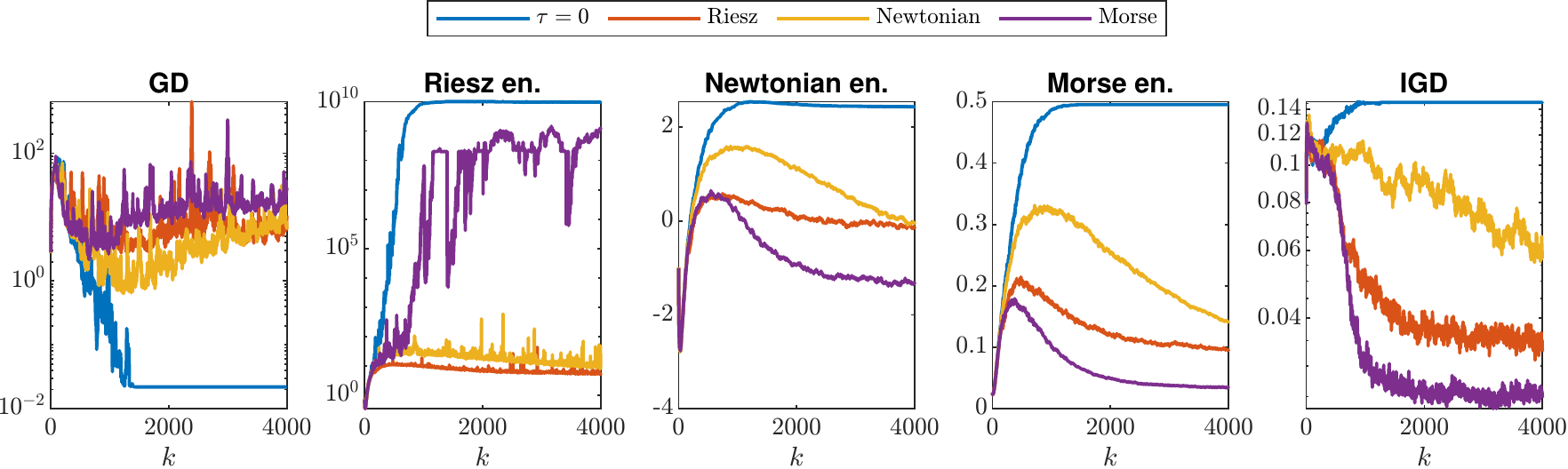}
\caption{Problem Lamé $\gamma = 0.25$}
\label{fig:2a}
\end{subfigure}
\begin{subfigure}{0.91\linewidth}
\centering
\includegraphics[trim = 0cm 0cm 0cm 0.5cm , clip, width=1\linewidth]{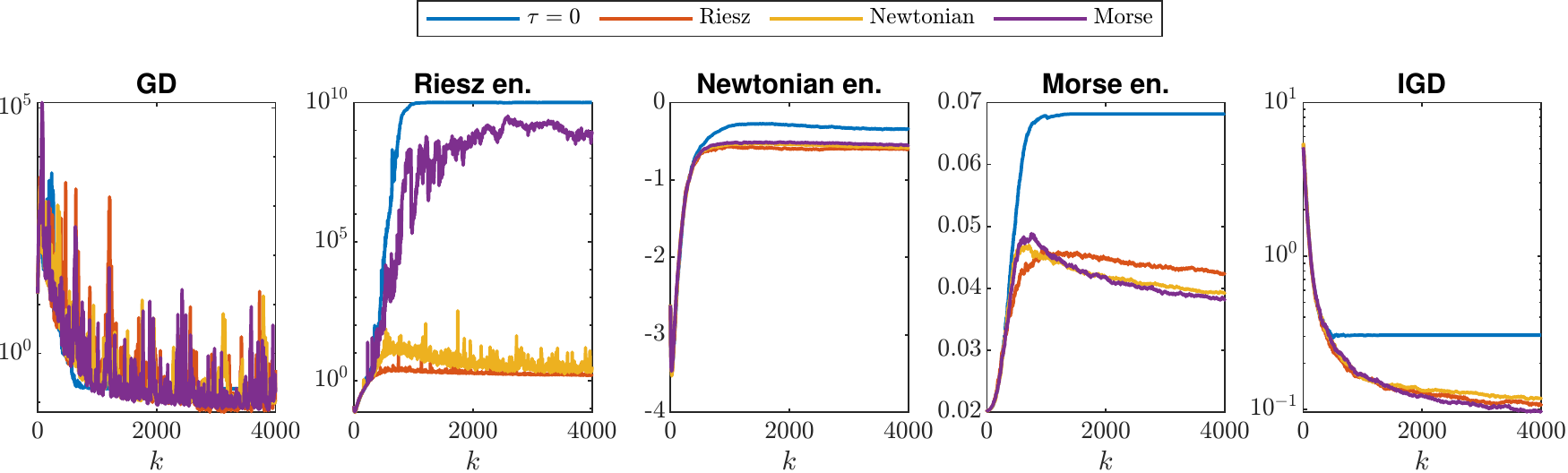}
\caption{Problem DO2DK $k=2, s=1$}
\label{fig:2b}
\end{subfigure}
\caption{Performance metric evolution, results are averaged over 25 runs.}
\label{fig:2}
\end{figure}

\subsection{Effect of the parameter $\tau$ and scalability}

By looking at the computational results, it becomes clear that the two phases of the algorithm, the one characterized by the CBO dynamics and the one characterized by the gradient-flows dynamics, have different scales. Typically, the former dynamics is much slower compared with the second one. This was consistent with assumptions to \cref{t:CBO}, where $\tau$ needs to be taken of order $o(\sqrt{\varepsilon})$. 

\begin{figure}
\centering
\begin{subfigure}{1\linewidth}
\includegraphics[trim = 0cm 0cm 0cm 0cm , clip, width=1\linewidth]{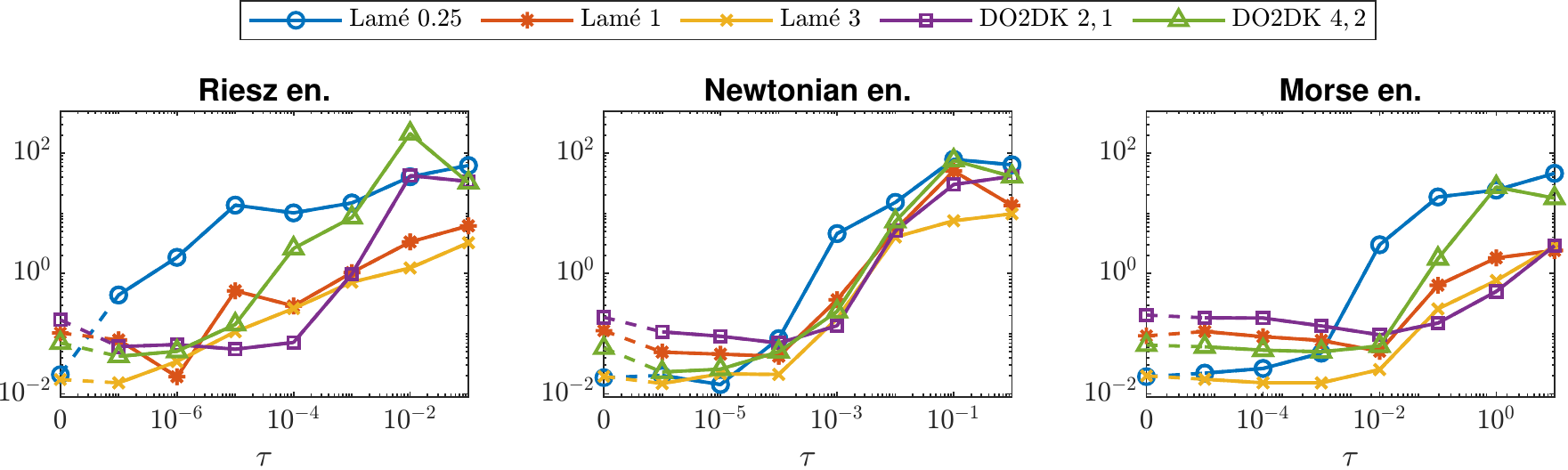}
\caption{Generational Distance (GD) \eqref{eq:GDnum}}
\label{fig:3a}
\end{subfigure}
\begin{subfigure}{1\linewidth}
\centering
\includegraphics[trim = 0cm 0cm 0cm 0.5cm , clip, width=1\linewidth]{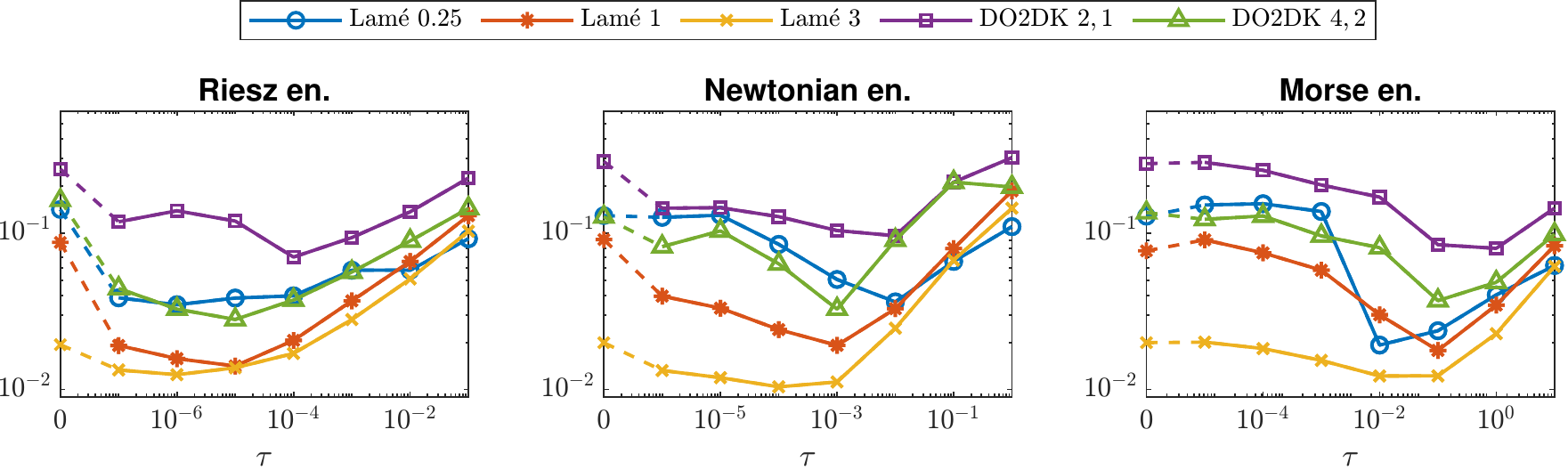}
\caption{Inverted Generational Distance (IGD) \eqref{eq:IGDnum}}
\label{fig:3b}
\end{subfigure}
\caption{Performance metrics as a function of $\tau$ for all the problems considered. Results are averaged over 10 runs.}
\label{fig:3}
\end{figure}

To experimentally investigate the importance, we test the algorithm for various values of $\tau$, keeping the remaining parameters fixed. \cref{fig:3a,fig:3b} show the final $GD$ and $IGD$ metrics when different binary potential are used during the computation.
As expected, relatively large values of $\tau$ lead to a strong interaction in parameter space that  interferes with the CBO mechanism. As a result, the $GD$ metric increases for large values of $\tau$. Interestingly, the lowest $GD$ values are not always attained for the smallest values of $\tau$, suggesting that the additional weights vectors dynamics might help the CBO mechanism in optimizing the sub-problems.

The $IGD$ metrics in \cref{fig:3b},shows that the optimal value of $\tau$ is different for each test case. In particular, DO2DK problems benefit from a strong interaction in parameter space. This might be explained by the front geometry (\cref{fig:1}):  the front length is long and, as consequence, the particles tend to be further apart in the image space, making the binary potential interaction weaker. Larger values of $\tau$ mitigate this effect, leading to better algorithm performances. If the extrema of the Pareto front are known in advance, one could address this issue by estimating the front length and choosing the parameter $\tau$ accordingly. We also note that algorithm seems to perform better when the Morse potential is used during the computation.

As already mentioned, the dynamics in $\Omega$ adds stochasticity to the particles position evolution. Hence,  the additional diffusive term $\sigma D_k^iB_k^i$ in \eqref{eq:iterx} might not be necessary. Yet,  taking $\sigma = 0$ yields  poor approximations of the Pareto front, see \cref{fig:4a}, suggesting that the diffusive term is still of paramount importance for the particles exploration behavior and their statistical independence.
From \cref{fig:4b}, it is obvious that the optimal diffusion parameter $\sigma$ is larger, the smaller $\tau$ is. In particular, if $\tau=0$ the particles diverge from the optimal EP points only when $\sigma>10$, which is consistent with other CBO methods for single-objective optimization, see for instance \cite{benfenati2021binary}. At the same time, for some problems, if $\sigma$ is too small, larger values of $\tau$ improve the convergence towards optimal points.

\begin{figure}
\centering

\begin{subfigure}{1.\linewidth}
\centering
\includegraphics[trim = 0cm 0cm 0cm 0cm , clip, width=1\linewidth]{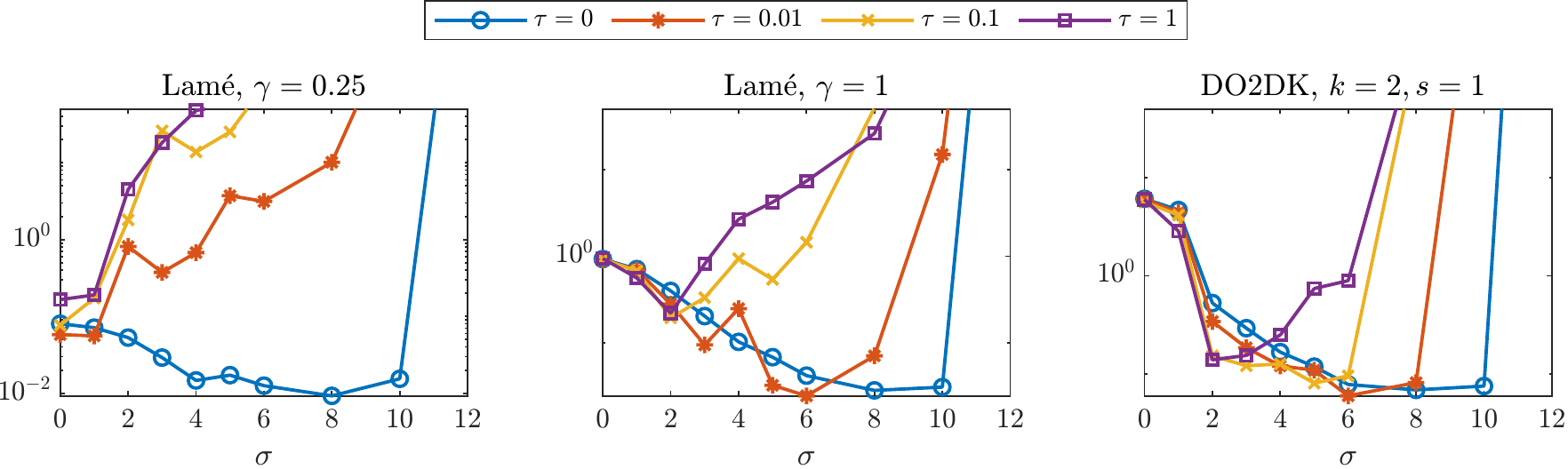}
\caption{$GD$}
\label{fig:4b}
\end{subfigure}
\begin{subfigure}{1.\linewidth}
\includegraphics[trim = 0cm 0cm 0cm 0.5cm , clip, width=1\linewidth]{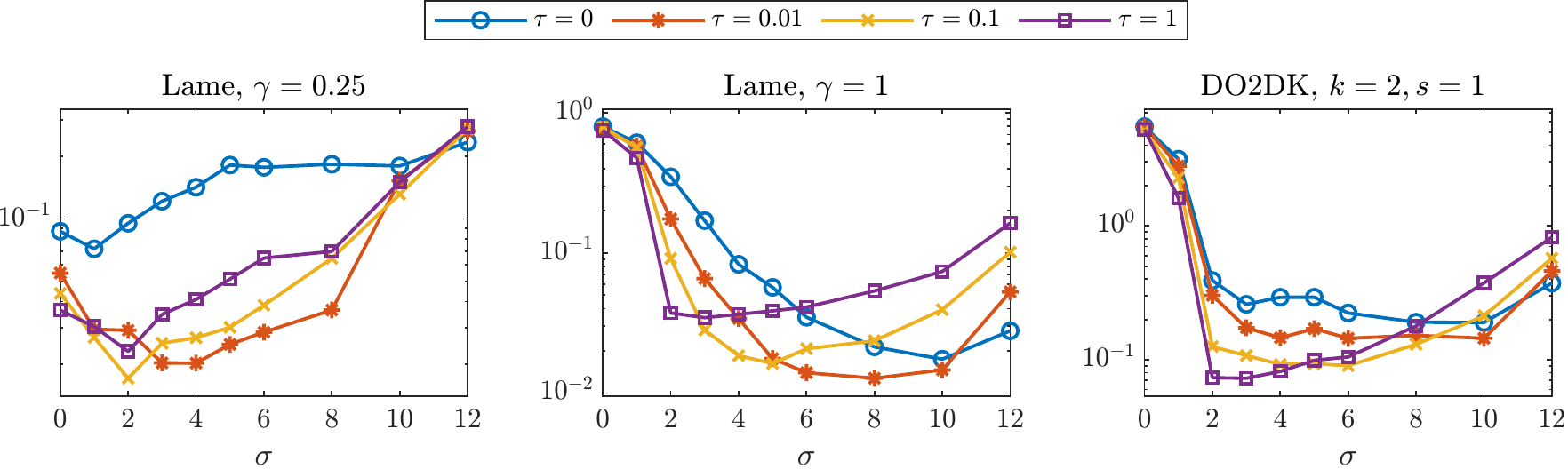}
\caption{$IGD$}
\label{fig:4a}
\end{subfigure}
\caption{Metrics as functions of $\sigma$, for different values of $\tau$. Morse interaction is used, results are averaged over 5 runs.}
\label{fig:4}
\end{figure}

Finally, we test the algorithm performance for different dimensions $d$ of the search space, keeping the same parameters choice. If the same number $N=100$ of particles are used, the $IGD$ of the computed solutions increases as the space dimension $d$ becomes larger, see \cref{fig:5}. This effect can be simply reduced by increasing the number of particles linearly with the space dimension, see \cref{fig:5}.

\begin{figure}
\centering
\includegraphics[trim = 0cm 0cm 0cm 0cm, clip, width=1\linewidth]{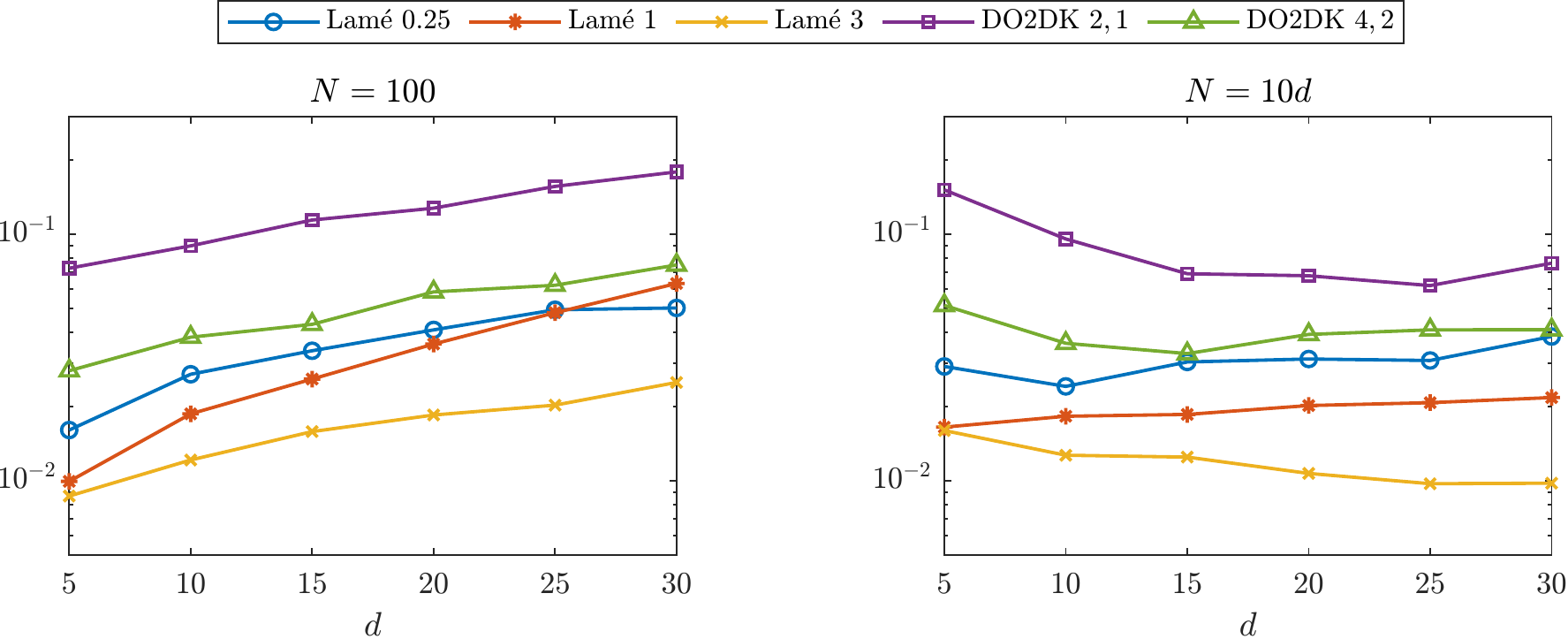}
\caption{$IGD$ metrics as functions of the search space dimension $d$. Morse interaction is used, results are averaged over 20 runs.}
\label{fig:5}
\end{figure}

\section{Conclusions}

In this work, we proposed an adaptive stochastic particle dynamics based on consensus to solve multi-objective optimization problems. The method makes use of a scalarization strategy the break down the original problem into $N$ parametrized single-objective sub-problems. 
The proposed algorithm, AM-CBO, extends prior work on multi-objective consensus based optimization by an additional adaptive dynamics in the  parameter space in order to  ensure that the particles distribute uniformly over the Pareto front. This is achieved by exploiting energy-based diversity measures. A rigorous mathematical analysis and numerical evidence are provided to validate this behavior. We theoretically investigated the long time behavior of the particle dynamics under the propagation of chaos assumption and establish convergence towards optimal points.  Indeed, under appropriate assumptions, the particles are capable of solving several single-objective problems at the same time, with a remarkable save of computational cost with the respect to a naive  approach. The additional dynamics on the parameter space is also analyzed based on  results on non-linear aggregation equations. Numerical experiments show that  the proposed method is capable to solve multi-objective problems with very different  Pareto fronts. The algorithm scales well with the problem dimension, even when using a relatively small number of particles.

\appendix
\section{Construction of reference solutions}
\label{app:1}
Even if we assume there exists an analytical representation of the Pareto front $F$, finding an $M$-approximation of $F$ which also minimizes a given two-body potential is a computationally expensive task, which is related to the already mentioned crystallization problem in physics \cite{blanc2015crystal}. In \cite{braun2015preference}, this was achieved by using mathematically programming techniques, while in \cite{coello2020reference} the authors proposed the following heuristic strategy: generate $N \gg M$ points on the front and
iteratively delete the point subject to the highest potential energy until only $M$ are left. In this appendix, we propose a different heuristic strategy which not only generates low-energy approximations of the front, but also provide more insight into the choice of the proposed update strategy \eqref{eq:iterw}.

In the following, we assume $F$ to be a $(m-1)$-dimensional manifold with known chart $\bar g \in \mathcal{C}^2(V, F)$
\[
F = \{ \bar g(z) \, |\, z \in V\}\,,
\]
where $V = [0,1]^{m-1}$ or $V = \Omega$. We also assume the tangential space $T(y, F)$ to be well-defined for all $y \in F$. Let $\{G_t^i \}_{i=1}^M \subset F$
describe the positions at time $t\geq 0$ of $M$ particles interacting over the front under a potential $U$, that is
\be
\frac{d}{dt}G_t^i = P_{T(G_t^i, F)} \left ( -\frac1M \sum_{j=1}^M \nabla U (G_t^i - G_t^j)  \right) \,,
\label{eq:a1}
\ee
with some given initial conditions $G_0^i = \bar g(Z_0^i)$, for all $i = 1, \dots, M$. 
Our heuristic strategy is based on the conjecture that, as $t \to \infty$, the system will eventually converge towards a low-energy configuration. Rather then solving \eqref{eq:a1} in $\RR^m$ where $F$ is embedded, we consider the equivalent
system for the coordinates $\{Z_t^i\}_{i=1}^M$,
\be
\frac{d}{dt}Z_t^i = \left( D \bar g(Z_t^i)\right)^+ P_{T(G_t^i, F)} \left ( -\frac1M \sum_{j=1}^M \nabla U (G_t^i - G_t^j) \right)\,,
\label{eq:a2}
\ee
where $G_t^i = \bar g(Z_t^i)$ and $(\cdot)^+$ is the pseudo-inverse of $(\cdot)$, see \cite[Chapter 5]{hairer2006geom}. We note that if $G_t^i$ belongs to the extrema of $F$ (or its ``contour'' when $m>2$), $D\bar g(Z_t^i)$ might not be well-defined. In this case, though, the projection is the null map so we set $dZ_t^i/dt = 0$. System \eqref{eq:a2} can then be solved numerically if $\bar g$ is explicitly known. The reference, low-energy, solution to \eqref{eq:mop} will then consist on the final configuration $\{G_T^i\}_{i=1}^M$ reached at certain time horizon $T>0$.

We note that dynamics \eqref{eq:sdew} introduced in the parameters space $\Omega$ can be seen as an approximation to \eqref{eq:a2} when $m=2$. Indeed, let $V = \Omega$ and the chart $\bar g$ be the relation (given by \cref{t:pareto}) between parameters and points on $F$
\[
\bar g(w) = \underset{x \in \RR^d}{\textup{argmin}} \,G(x,w)\,.
\]
As $\bar g$ and $T(y,F)$ in \eqref{eq:a2} are unknown during the optimization process, one could approximate them by assuming linearity on $F$. In particular, if no further information on the front geometry is available, let us take $F = \Omega$ as in \cref{a:4}. This leads to 
\[
\bar g(w) \approx A w, \quad Dg(w) \approx A \quad \textup{where}\quad  A = \begin{pmatrix}
0 & 1\\
1 & 0 
\end{pmatrix}\,,
\] 
as before in \cref{l:1}, and
\[ P_{T(\bar g(w), F)} \approx P_{T(Aw,\Omega)} = \begin{pmatrix}
1 & -1\\
-1 & 1 
\end{pmatrix} = :B
\]
if $w>0$ component-wise, and $P_{T(Aw,\Omega)} = 0$ otherwise. Starting from \eqref{eq:a2}, it follows
\be
\frac{d}{dt}W_t^i = \left( D \bar g(W_t^i)\right)^+ P_{T(G_t^i, F)} \left ( -\frac1M \sum_{j=1}^M \nabla U (G_t^i - G_t^j) \right)
 \approx A^+ B \left(  - \frac1{M}\sum_{j=1}^M \nabla U (G_t^i - G_t^j)\right) \,.
 \notag
 \ee
Now, since $A^+ = A$ and $AB = -B$ we obtain
\be
\frac{d}{dt}W_t^i \approx - P_{T(w,\Omega)} \left ( -\frac1M \sum_{j=1}^M \nabla U (G_t^i - G_t^j) \right)
\label{eq:a3}
\ee
which corresponds to the dynamics proposed in \cref{sec:3}, provided $G_t^i \approx g(X_t^i)$.

To conclude, we remark that the above approximation has a mild impact on the final distribution over the front, even when $F$ differs substantially from $\Omega$, see \cref{fig:6}.

\begin{figure}
\begin{subfigure}{0.49\linewidth}
\centering
\includegraphics[trim = 0cm 0cm 0cm 0cm , clip, width=1\linewidth]{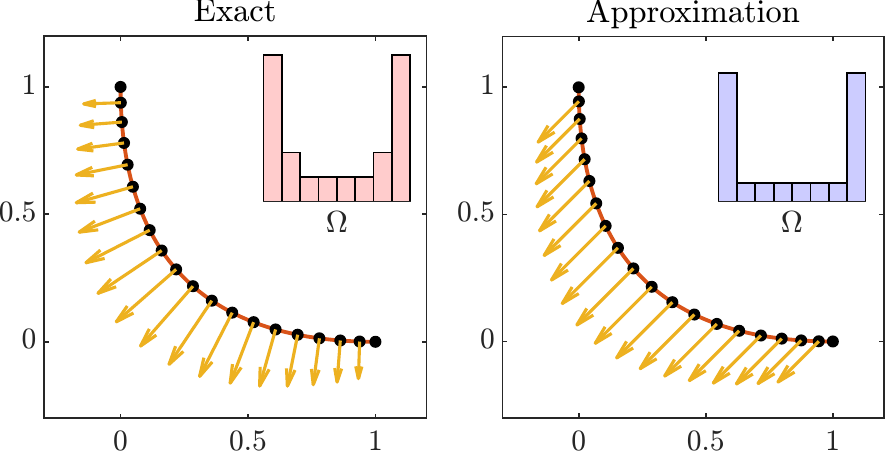}
\caption{Problem Lamé $\gamma = 0.5$}
\label{fig:4b}
\end{subfigure}
\hfill
\begin{subfigure}{0.49\linewidth}
\includegraphics[trim = 0cm 0cm 0cm 0cm , clip, width=1\linewidth]{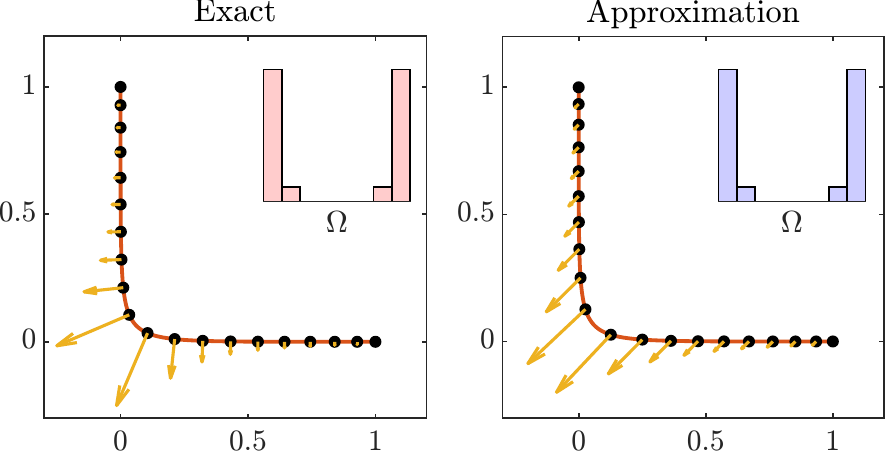}
\caption{Problem Lamé $\gamma = 0.25$}
\label{fig:4a}
\end{subfigure}
\caption{Final configuration of $M = 20$ particles evolved according to \eqref{eq:a2} and \eqref{eq:a3} under Riesz binary potential \eqref{eq:riesz}. Arrows show the total potential forces the particles are subject to. Two different front shapes are considered. The particles systems are solved with an explicit Euler scheme, $\Delta t = 10^{-8}, T = 0.01$. Histograms show the final distribution over the coordinate space $V=\Omega$.}
\label{fig:6}
\end{figure}

\section{Problem definition}
\label{app:2}

We report here the problems definition, together with the penalization strategy and known
parametrization of $F$.
The Lamé \cite{emmerich2007lame} and the DO2DK \cite{branke2004finding} problems are originally formulated as constrained multi-objective optimization problems where the feasible domain is given by $\mathcal{H} = [0,1]^d$. Moreover the set of EP optimal points corresponds the edge $[0,1]\times\{0\}^{d-1}$. Adding a projection step to $\mathcal{H}$ has a relevant impact on the algorithm dynamics, as any point belonging to the cone $\RR \times \RR^{d-1}_{\leq0}$ is projected to an EP optimal point. Therefore, we make use of an exact penalization strategy to ensure the particles remain the feasible region adding a $\ell_1$-penalty term of the form $\beta \dist(x, \mathcal{H})$, $\beta>0$, to the original objective functions.

Let $x \in \RR^d, x = (x_1, \dots, x_d)$ for $d\geq 1$, the objective functions are given by
\begin{itemize}
\item Lamé \cite{emmerich2007lame} with $\gamma \in \RR_{>0}$, 
\be
\begin{split}
g_1(x) &= \left | \cos\left(\frac\pi2 x_1 \right) \right|^{\frac2\gamma} \left( 1+ r(x) \right) + \frac\pi\gamma \textup{dist}(x, \mathcal{H}) \\
g_2(x) &= \left | \sin\left(\frac\pi2 x_1 \right)\cos\left(\frac\pi2 x_2 \right) \right|^{\frac2\gamma} \left( 1+ r(x) \right) + \frac\pi\gamma \textup{dist}(x, \mathcal{H})\\
\end{split}
\ee
with
$r(x) = \sqrt{\sum_{i=2}^d x_i^2} $.
\item DO2DK \cite{branke2004finding} with $k \in \mathbb{N}, s \in \RR_{>0}$
\be
\begin{split}
g_1(x) &=\sin \left(\frac\pi2 x_1 + \left(1+ \frac{2^s -1}{2^{s+2}} \right) \pi +1 \right)r_a(x) r_b(x) + 10\textup{dist}(x, \mathcal{H}) \\
g_2(x) & = \left( \cos\left(\frac \pi 2 x_1 +\pi \right)+1 \right) r_a(x) r_b(x)  + 10\textup{dist}(x, \mathcal{H})
\end{split}
\ee
with
\begin{align*}
r_a(x) &= 1 + \frac{9}{d-1}\sum_{i=2}^d x_i \\
r_b(x) &= 5 + 10\left(x_1 -\frac12\right)^2 + \frac{2^{\frac s2}\cos(2k\pi x_1)}k \,.
\end{align*}

\end{itemize}
The parametrization used to construct reference solutions is given by
\[
h:  [0,1] \rightarrow F \, , \quad h(r) = g\left( (r,0, \dots,0)\right) \in F\,.
\]

\bigskip
{\small {\bf Acknowledgments}
This work has been written within the
activities of GNCS group of INdAM (National Institute of
High Mathematics). L.P. acknowledges the partial support of MIUR-PRIN Project 2017, No. 2017KKJP4X “Innovative numerical methods for evolutionary partial differential equations and applications”. 
The work of G.B. is funded by the Deutsche Forschungsgemeinschaft (DFG, German Research Foundation) – Projektnummer 320021702/GRK2326 – Energy, Entropy, and Dissipative Dynamics (EDDy). 
M.H.  thanks the Deutsche Forschungsgemeinschaft (DFG, German Research Foundation) for the financial support through 320021702/GRK2326,  333849990/IRTG-2379, CRC1481, HE5386/18-1,19-2,22-1,23-1, ERS SFDdM035 and under Germany’s Excellence Strategy EXC-2023 Internet of Production 390621612 and under the Excellence Strategy of the Federal Government and the Länder. The authors acknowledge the support of the Banff International Research Station (BIRS) for the Focused Research Group [22frg198] “Novel perspectives in kinetic equations for emerging phenomena”, July 17-24, 2022, where part of this work was done.
}

\bibliographystyle{abbrv}
\bibliography{bibfile}

\end{document}